\documentclass[twoside,11pt]{article}
\usepackage{isipta}

\usepackage[utf8]{inputenc}
\inputencoding{utf8} 

\usepackage{hyperref,color,soul,booktabs}
\setulcolor{blue}

\usepackage{cmap} 
\usepackage[british]{babel}
\usepackage[T1]{fontenc}
\usepackage{amssymb,bm}
\usepackage{mathtools}
\usepackage{nicefrac}
  \mathtoolsset{mathic}
  \allowdisplaybreaks
\usepackage{courier,mathptmx,stmaryrd}
\usepackage{longtable}
\usepackage{hyperref,url}
  \hypersetup{  
    bookmarksnumbered
  }
  \hypersetup{  
    breaklinks=false,
    pdfborderstyle={/S/U/W 0},  
    citebordercolor=.235 .702 .443,
    urlbordercolor=.255 .412 .882,
    linkbordercolor=.804 .149 .19,
  }
  \hypersetup{ 
    pdfauthor={Gert de Cooman and Jasper De Bock},
    pdftitle={Computable randomness is inherently imprecise},
    pdfkeywords={computable randomness, imprecise probabilities, game-theoretic probability, supermartingale}
  }
\usepackage[all]{hypcap} 
\usepackage{xcolor}
\usepackage{tikz}
  \usetikzlibrary{calc,matrix,plotmarks,intersections,arrows}
  \tikzstyle{root}=[rectangle,draw=blue!90]
  \tikzstyle{nonterminal}=[rectangle,rounded corners,fill=blue!15,draw=blue!15]
  \tikzstyle{terminal}=[rectangle]
  \tikzstyle{cut}=[thick,dotted,draw=green!50!black]
  \tikzstyle{local}=[color=green!50!black,text=green!25!black]
\usepackage{enumitem}
\DeclarePairedDelimiter{\group}{(}{)}
\DeclarePairedDelimiter{\sqgroup}{[}{]}
\DeclarePairedDelimiter{\set}{\{}{\}}
\DeclarePairedDelimiter{\avg}{\llbracket}{\rrbracket}

\DeclarePairedDelimiter{\norm}{\Vert}{\Vert}
\DeclarePairedDelimiter{\abs}{\vert}{\vert}

\newcommand{\frcsts}{\sqgroup{0,1}}

\newcommand{\imprecisefrcsts}{\mathcal{C}}

\newcommand{\outcomes}{\set{0,1}}
\newcommand{\pths}{\Omega}
\newcommand{\sits}{\pths^\lozenge}

\newcommand{\constantfrcst}[1]{\gamma_{\,#1}}
\newcommand{\average}[2][S]{\avg{#2}_{#1}}

\newcommand{\frcstsystem}{\gamma}
\newcommand{\lfrcstsystem}{\underline{\frcstsystem}}
\newcommand{\ufrcstsystem}{\overline{\frcstsystem}}
\newcommand{\xmplfrcstsystem}{{\frcstsystem_{\sim\nicefrac{1}{2}}}}
\newcommand{\prcsfrcstsystem}[1][p]{{\frcstsystem_{\set{#1}}}}

\newcommand{\frcstsystems}{\Gamma}
\newcommand{\constantfrcstsystems}{\imprecisefrcsts}
\newcommand{\process}{F}
\newcommand{\processtoo}{G}
\newcommand{\processalso}{H}
\newcommand{\selection}{S}
\newcommand{\multprocess}{D}
\newcommand{\mint}[1][\multprocess]{#1^\circledcirc}
\newcommand{\supermartin}{M}
\newcommand{\submartin}{M}
\newcommand{\submartins}[1][\frcstsystem]{\underline{\mathbb{M}}^{#1}}
\newcommand{\supermartins}[1][\frcstsystem]{\overline{\mathbb{M}}^{#1}}
\newcommand{\testsupermartins}[1][\frcstsystem]{\overline{\mathbb{T}}^{#1}}
\newcommand{\comptestsupermartins}[1][\frcstsystem]{\overline{\mathbb{T}}^{#1}_{\mathrm{C}}}
\newcommand{\test}{T}

\newcommand{\random}[2][\pth]{\frcstsystems_{#2}(#1)}
\newcommand{\constantrandom}[2][\pth]{\constantfrcstsystems_{#2}(#1)}
\newcommand{\lowerconstantrandom}[2][\pth]{L_{#2}(#1)}
\newcommand{\upperconstantrandom}[2][\pth]{U_{#2}(#1)}

\newcommand{\lowersmallestrandom}[2][\pth]{\lp[#2](#1)}
\newcommand{\uppersmallestrandom}[2][\pth]{\up[#2](#1)}
\newcommand{\varnorm}[1]{\norm{#1}_\mathrm{v}}

\newcommand{\naturals}{\mathbb{N}}
\newcommand{\naturalswithzero}{\mathbb{N}_0}

\newcommand{\reals}{\mathbb{R}}

\newcommand{\ex}{E}
\newcommand{\lex}{\underline{\ex}}
\newcommand{\uex}{\overline{\ex}}

\newcommand{\ljoint}[1][]{\lex^{#1}}
\newcommand{\ujoint}[1][]{\uex^{#1}}

\newcommand{\lp}[1][]{\underline{p}_{#1}}
\newcommand{\up}[1][]{\overline{p}_{#1}}

\newcommand{\pinterval}[1][]{\sqgroup{\lp[#1],\up[#1]}}

\newcommand{\lptoo}[1][]{\underline{q}_{#1}}
\newcommand{\uptoo}[1][]{\overline{q}_{#1}}

\newcommand{\qinterval}[1][]{\sqgroup{\lptoo[#1],\uptoo[#1]}}

\newcommand{\init}{\square}
\newcommand{\pth}{\omega}

\newcommand{\randout}[1][]{X_{#1}}
\newcommand{\precedes}{\sqsubseteq}

\newcommand{\xval}[1]{x_{#1}}
\newcommand{\zval}[1]{z_{#1}}
\newcommand{\xvalto}[1]{\xval{1},\dots,\xval{#1}}
\newcommand{\zvalto}[1]{\zval{1},\dots,\zval{#1}}

\newcommand{\capitalsymbol}{\mathcal{K}}
\newcommand{\pmass}[1][]{\widehat{p}_{#1}}

\newcommand{\comp}{computable}
\newcommand{\scomp}{semicomputable}
\newcommand{\lscomp}{lower semicomputable}
\newcommand{\uscomp}{upper semicomputable}

\newcommand{\cset}[3][]{\set[#1]{#2\colon#3}}
\newcommand{\ind}[1]{\mathbb{I}_{#1}}
\newcommand{\indsing}[1]{\ind{\set{#1}}}
\newcommand{\then}{\Rightarrow}

\newcommand{\ifandonlyif}{\Leftrightarrow}

\newcommand{\adddelta}{\Delta}

\ShortHeadings{Computable Randomness is Inherently Imprecise}{De Cooman and De Bock} 

\begin{document}

\title{Computable Randomness is Inherently Imprecise}

\author{\name Gert de Cooman \email gert.decooman@ugent.be\\
\name Jasper De Bock \email jasper.debock@ugent.be\\
\addr Ghent University, IDLab\\
Technologiepark -- Zwijnaarde 914, 9052 Zwijnaarde, Belgium}
\maketitle
\begin{abstract}
We use the martingale-theoretic approach of game-theoretic probability to incorporate imprecision into the study of randomness.
In particular, we define a notion of computable randomness associated with interval, rather than precise, forecasting systems, and study its properties. 
The richer mathematical structure that thus arises lets us better understand and place existing results for the precise limit.
When we focus on constant interval forecasts, we find that every infinite sequence of zeroes and ones has an associated filter of intervals with respect to which it is computably random. 
It may happen that none of these intervals is precise, which justifies the title of this paper. 
We illustrate this by showing that computable randomness associated with non-stationary precise forecasting systems can be captured by a stationary interval forecast, which must then be less precise: a gain in model simplicity is thus paid for by a loss in precision.
\end{abstract}

\begin{keywords}
computable randomness; imprecise probabilities; game-theoretic probability; interval forecast; supermartingale; computability.
\end{keywords}

\section{Introduction}\label{sec:introduction}
This paper documents the first steps in our attempt to incorporate indecision and imprecision into the study of randomness.
Consider a infinite sequence $\pth=(\zvalto{n},\dots)$ of zeroes and ones; when do we call it \emph{random}?
There are many notions of randomness, and many of them have a number of equivalent definitions \citep{ambosspies2000,bienvenu2009:randomness}.
We focus here on \emph{computable randomness}, mainly because its focus on computability---rather than, say, the weaker lower semicomputability---has allowed us in this first attempt to keep the mathematical nitpicking at arm's length.
Randomness of a sequence $\pth$ is typically associated with a probability measure on the sample space of all infinite sequences, or---what is equivalent---with a \emph{forecasting system} $\frcstsystem$ that associates with each finite sequence of outcomes $(\xvalto{n})$ the (conditional) expectation $\frcstsystem(\xvalto{n})$ for the next (as yet unknown) outcome $\randout[n+1]$.
The sequence $\pth$ is then called \emph{computably} random when it passes a (countable) number of \emph{computable} tests of randomness, where the collection of randomness tests depends of the forecasting system $\frcstsystem$.
An alternative but equivalent definition, going back to \citet{ville1939}, sees each forecast $\frcstsystem(\xvalto{n})$ as a fair price for---and therefore a commitment to bet on---the as yet unknown next outcome $\randout[n+1]$.
The sequence $\pth$ is then computably random when there is no computable strategy for getting infinitely rich by exploiting the bets made available by the forecasting system $\frcstsystem$ along the sequence, without borrowing.
Technically speaking, all computable non-negative supermartingales should remain bounded on $\pth$, and the forecasting system $\frcstsystem$ determines what a supermartingale is.

It is this last, martingale-theoretic approach which seems to lend itself most easily to allowing for imprecision in the forecasts, and therefore in the definition of randomness.
As we explain in Sections~\ref{sec:single:forecast} and~\ref{sec:forecasting:systems}, an `imprecise' forecasting system $\frcstsystem$ associates with each finite sequence of outcomes $(\xvalto{n})$ a (conditional) expectation \emph{interval} $\frcstsystem(\xvalto{n})$ for the next (as yet unknown) outcome $\randout[n+1]$, whose lower bound represents a supremum acceptable buying price, and whose upper bound a infimum acceptable selling price for $\randout[n+1]$.
This idea rests firmly on the common ground between \citeauthor{walley1991}'s \citeyearpar{walley1991} theory of coherent lower previsions and \citeauthor{shafer2001}'s \citeyearpar{shafer2001} game-theoretic approach to probability that we have established in recent years, through our research on imprecise stochastic processes \citep{cooman2007d,cooman2015:markovergodic}.
This allows us to associate supermartingales with an imprecise forecasting system, and therefore in Section~\ref{sec:randomness} to extend the existing notion of computable randomness to allow for interval, rather than precise, forecasts---we discuss computability in Section~\ref{sec:computability}.
We show in Section~\ref{sec:consistency} that our approach allows us to extend some of \citeauthor{dawid1982:well:calibrated:bayesian}'s \citeyearpar{dawid1982:well:calibrated:bayesian} well-known work on calibration, as well as an interesting `limiting frequencies' or computable stochasticity result.
 
We believe the discussion becomes really interesting in Section~\ref{sec:constantintervalforecasts}, where we look at stationary interval forecasts to extend the classical account of randomness. 
That classical account typically considers a forecasting system with stationary expectation forecast $\nicefrac{1}{2}$---corresponding to flipping a fair coin. 
As we have by now come to expect from our experience with imprecise probability models, a much more interesting mathematical picture appears when allowing for interval forecasts than the rather simple case of precise forecasts would lead us to suspect.
In the precise case, a given sequence may not be (computably) random for any stationary forecast, but in the imprecise case there is always a set filter of intervals that a given sequence is computably random for. 
Furthermore, as we show in Section~\ref{sec:inherently}, this filter may not have a smallest element, and even when it does, this smallest element may be a non-vanishing interval: randomness may be inherently imprecise.

\section{A single interval forecast}
\label{sec:single:forecast}
The dynamics of making a single forecast can be made very clear by considering a simple game, with three players, namely Forecaster, Sceptic and Reality.

\vspace{.5\baselineskip}\noindent{\bf Game: single forecast of an outcome $\randout$}\newline
In a first step, Forecaster specifies an interval bound $I=\sqgroup{\lp,\up}$ for the expectation of an as yet unknown outcome $\randout$ in $\outcomes$---or equivalently, for the probability that $\randout=1$. 
We interpret this \emph{interval forecast} $I$ as a commitment, on the part of Forecaster, to adopt $\lp$ as a \emph{supremum buying price} and $\up$ as a \emph{infimum selling price} for the gamble (with reward function) $\randout$.
This is taken to mean that the second player, \emph{Sceptic}, can now in a second step take Forecaster up on any (combination) of the following commitments:
\begin{enumerate}[label=\upshape(\roman*),leftmargin=*,noitemsep,topsep=0pt]
\item for any $p\in\frcsts$ such that $p\leq\lp$, and any $\alpha\geq0$ Forecaster must accept the gamble $\alpha[\randout-p]$, leading to an uncertain reward $-\alpha[\randout-p]$ for Sceptic;\footnote{Because we allow $p\leq\lp$ rather than $p<\lp$, we actually see $\lp$ as a \emph{maximum} buying price, rather than a supremum one. We do this because it does not affect the conclusions, but simplifies the mathematics. Similarly for $q\geq\up$.}
\item for any $q\in\frcsts$ such that $q\geq\up$, and any $\beta\geq0$ Forecaster accepts the gamble $\beta[q-\randout]$, leading to an uncertain reward $-\beta[q-\randout]$ for Sceptic.
\end{enumerate}
Finally, in a third step, the third player, \emph{Reality}, determines the value $x$ of $\randout$ in $\outcomes$.\hfill$\square$\par\vspace{.5\baselineskip}

Elements $x$ of $\outcomes$ are called \emph{outcomes}, and elements $p$ of the real unit interval $\frcsts$ are called (precise) \emph{forecasts}.
We denote by $\imprecisefrcsts$ the set of non-empty closed subintervals of the real unit interval $\frcsts$.
Any element $I$ of $\imprecisefrcsts$ is called an \emph{interval forecast}. 
It has a smallest element $\min I$ and a greatest element $\max I$, so $I=\sqgroup{\min I,\max I}$.
We will use the generic notation $I$ for such an interval, and $\lp\coloneqq\min I$ and $\up\coloneqq\max I$ for its lower and upper bounds, respectively.

After Forecaster announces a forecast interval $I$, what Sceptic can do is essentially to try and increase his capital by taking a gamble on the outcome $\randout$. 
Any such gamble can be considered as a map $f\colon\outcomes\to\reals$, and can therefore be represented as a vector $(f(1),f(0))$ in the two-dimensional vector space $\reals^2$; see also Figure~\ref{fig:capitalincrease}.
\begin{figure}[t]
\centering
\begin{tikzpicture}[scale=1.75]
\coordinate (xaxis) at (1.5,0);
\coordinate (yaxis) at (0,1.5);
\coordinate (origin) at (0,0);
\coordinate (bottomleft) at (-1.5,-1.5);
\coordinate (aboveleft) at (-1.5,0.5);
\coordinate (belowright) at (0.5,-1.5);
\coordinate (gamble) at (-1,-0.5);
\coordinate (label) at (0,-1.7);
\draw[blue,thick,name path=marginal line] (aboveleft) -- node[midway,above,rotate=-18.43] {\footnotesize $\ex_{\lp}(f)=0$} (origin);
\draw[help lines,dashed] (origin) -- (1.5,-0.5);
\draw[help lines,dashed] (-0.5,1.5) -- (origin);
\draw[blue,thick] (origin) -- node[midway,above,rotate=-71.57] {\footnotesize $\ex_{\up}(f)=0$} (belowright);
\fill[nearly transparent,blue] (origin) -- (belowright) -- (bottomleft) -- (aboveleft) -- cycle;
\draw[step=.5cm,gray,very thin] (-1.4,-1.4) grid (1.4,1.4); 
\draw[->] (-1.5,0) -- (xaxis) node[below] {\footnotesize$f(1)$};
\draw[->] (0,-1.5) -- (yaxis) node[above] {\footnotesize$f(0)$};
\draw[red,very thin] (bottomleft) -- (origin) -- coordinate[midway] (diagonal) (1.5,1.5) ; 
\node[red,rotate=45,above=5pt] at (diagonal) {\footnotesize $f(1)\leq f(0)$};
\node[red,rotate=45,below=5pt] at (diagonal) {\footnotesize $f(1)\geq f(0)$};
\node[below=-10pt] at (label) {(a)};
\end{tikzpicture}
\quad
\begin{tikzpicture}[scale=1.75]
\coordinate (xaxis) at (1.5,0);
\coordinate (yaxis) at (0,1.5);
\coordinate (origin) at (0,0);
\coordinate (bottomleft) at (-1.5,-1.5);
\coordinate (bottomright) at (1.5,-1.5);
\coordinate (aboveleft) at (-1.5,0.5);
\coordinate (belowright) at (1.5,-0.5);
\coordinate (gamble) at (-1,-0.5);
\coordinate (label) at (0,-1.7);
\draw[step=.5cm,gray,very thin] (-1.4,-1.4) grid (1.4,1.4); 
\draw[->] (-1.5,0) -- (xaxis) node[below] {\footnotesize$f(1)$};
\draw[->] (0,-1.5) -- (yaxis) node[above] {\footnotesize$f(0)$};
\draw[red,very thin] (bottomleft) -- (origin) -- coordinate[midway] (diagonal) (1.5,1.5) ; 
\node[red,rotate=45,above=5pt] at (diagonal) {\footnotesize $f(1)\leq f(0)$};
\node[red,rotate=45,below=5pt] at (diagonal) {\footnotesize $f(1)\geq f(0)$};
\fill[nearly transparent,blue] (aboveleft) -- (belowright) -- (bottomright) -- (bottomleft) -- cycle;
\draw[fill=white,blue,thick,name path=marginal line] (aboveleft) -- node[midway,above,rotate=-18.43] {\footnotesize $\ex_{r}(f)=0$} (origin) -- (belowright);
\node[below=-10pt] at (label) {(b)};
\end{tikzpicture}
\caption{Gambles $f$ available to Sceptic when (a) Forecaster announces $I\in\imprecisefrcsts$ with $\lp<\up$; and when (b) Forecaster announces $I\in\imprecisefrcsts$ with $\lp=\up\eqqcolon r$.}
\label{fig:capitalincrease}
\end{figure}
$f(\randout)$ is then the increase in Sceptic's capital after the game has been played, as a function of the outcome variable $\randout$.
Of course, not every gamble $f(\randout)$ on the outcome $\randout$ will be available to Sceptic: which gambles he can take is determined by Forecaster's interval forecast $I$.
In their most general form, they are given by $f(\randout)=-\alpha[\randout-p]-\beta[q-\randout]$, where $\alpha$ and $\beta$ are non-negative real numbers, $p\leq\lp$ and $q\geq\up$.
If we consider the so-called \emph{lower expectation} (functional) $\lex_I$ associated with an interval forecast $I$, defined by 
\begin{equation}\label{eq:local:lower}
\lex_I(f)
=\min_{p\in I}\ex_p(f)
=\min_{p\in I}\sqgroup[\big]{pf(1)+(1-p)f(0)}
=
\begin{cases}
\ex_{\lp}(f)&\text{ if $f(1)\geq f(0)$}\\
\ex_{\up}(f)&\text{ if $f(1)\leq f(0)$}
\end{cases}
\end{equation}
for any gamble $f\colon\outcomes\to\reals$, and similarly, the \emph{upper expectation} (functional) $\uex_I$, defined by
\begin{equation}\label{eq:local:upper}
\uex_I(f)
=\max_{p\in I}\ex_p(f)
=\begin{cases}
\ex_{\up}(f)&\text{ if $f(1)\geq f(0)$}\\
\ex_{\lp}(f)&\text{ if $f(1)\leq f(0)$}
\end{cases}
=-\lex_I(-f),
\end{equation}
then it is not difficult to see that \emph{the cone of gambles $f(\randout)$ that are available to Sceptic after Forecaster announces an interval forecast $I$ is completely determined by the condition $\uex_I(f)\leq0$}, as depicted by the blue regions in Figure~\ref{fig:capitalincrease}.
The functionals $\lex_I$ and $\uex_I$ are easily shown to have the following properties, typical for the more general lower and upper expectation operators defined on more general gamble spaces \citep{walley1991,troffaes2013:lp}:

\begin{proposition}\label{prop:properties:of:expectations}
Consider any forecast interval $I\in\imprecisefrcsts$.
Then for all gambles $f,g$ on $\outcomes$, $\mu\in\reals$ and non-negative $\lambda\in\reals$:
\begin{enumerate}[label=\upshape C{\arabic*}.,ref=\upshape C{\arabic*},leftmargin=*,noitemsep,topsep=0pt]
\item\label{axiom:coherence:bounds} $\min f\leq\lex_I(f)\leq\uex_I(f)\leq\max f$;\upshape\hfill[bounds]
\item\label{axiom:coherence:homogeneity} $\lex_I(\lambda f)=\lambda\lex_I(f)$ and $\uex_I(\lambda f)=\lambda\uex_I(f)$;\upshape\hfill[non-negative homogeneity]
\item\label{axiom:coherence:subadditivity} $\lex_I(f+g)\geq\lex_I(f)+\lex_I(g)$ and $\uex_I(f+g)\leq\uex_I(f)+\uex_I(g)$;\upshape\hfill[super/subadditivity]
\item\label{axiom:coherence:constantadditivity} $\lex_I(f+\mu)=\lex_I(f)+\mu$ and $\uex_I(f+\mu)=\uex_I(f)+\mu$.\upshape\hfill[constant additivity]
\end{enumerate}
\end{proposition}

\section{Interval forecasting systems and imprecise probability trees}
\label{sec:forecasting:systems}
We now consider a sequence of repeated versions of the forecast game in the previous section, where at each stage $k\in\naturals$, Forecaster presents an interval forecast $I_k=\pinterval[k]$ for the unknown outcome variable $\randout[k]$.
This effectively allows Sceptic to choose any gamble $f_k(\randout[k])$ such that $\uex_{I_k}(f)\leq0$.
Reality then chooses a value $x_k$ for $X_k$, resulting in a gain, or increase in capital, $f_k(x_k)$ for Sceptic.

We call $(x_1,x_2,\dots,x_n,\dots)$ an \emph{outcome sequence}, and collect all possible outcome sequences in the set $\pths\coloneqq\outcomes^\naturals$.
We collect the finite outcome sequences $(x_1,\dots,x_n)$ in the set $\sits\coloneqq\outcomes^*=\bigcup_{n\in\naturalswithzero}\outcomes^n$.
Finite sequences $s$ in $\sits$ and infinite sequences $\pth$ in $\pths$ are the nodes---called \emph{situations}---and paths in an event tree with unbounded horizon, part of which is depicted below.
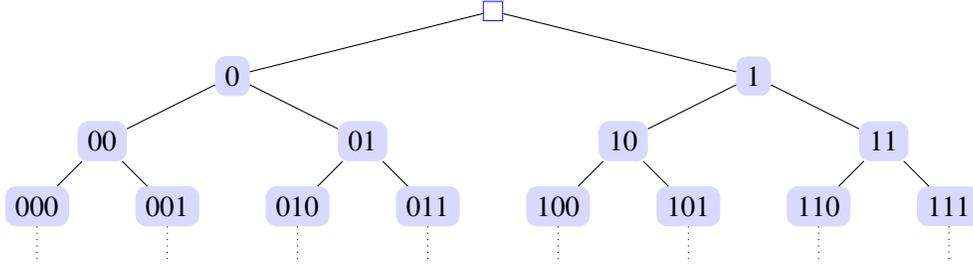
\begin{center}
\begin{tikzpicture}
\tikzstyle{level 1}=[sibling distance=18em]
\tikzstyle{level 2}=[sibling distance=9em]
\tikzstyle{level 3}=[sibling distance=4.5em]
\tikzstyle{level 4}=[level distance=2em]
\node[root] (root) {} [grow=down,level distance=5ex]
child {node[nonterminal] (a) {$0$}
  child {node[nonterminal] (aa) {$00$}
    child {node[nonterminal] (aaa) {$000$}
      child[black,dotted]} 
    child {node[nonterminal] (aab) {$001$}
      child[black,dotted]}
  }
  child {node[nonterminal] (ab) {$01$}
    child {node[nonterminal] (aba) {$010$}
      child[black,dotted]}
    child {node[nonterminal] (abb) {$011$}
      child[black,dotted]}
  }
}
child {node[nonterminal] (b) {$1$}
  child {node[nonterminal] (ba) {$10$}
    child {node[nonterminal] (baa) {$100$}
      child[black,dotted]}
    child {node[nonterminal] (bab) {$101$}
      child[black,dotted]}
  }
  child {node[nonterminal] (bb) {$11$}
    child {node[nonterminal] (bba) {$110$}
      child[black,dotted]}
    child {node[nonterminal] (bbb) {$111$}
      child[black,dotted]}
  }
};
\end{tikzpicture}

\end{center}
In this repeated game, Forecaster will only provide interval forecasts $I_k$ after observing the actual sequence $(x_1,\dots,x_{k-1})$ that Reality has chosen. 
This is the essence of so-called \emph{prequential forecasting} \citep{dawid1982:well:calibrated:bayesian,dawid1984,dawid1999}.
But for technical reasons, it will be useful to consider the more involved setting where a forecast $I_s$ is specified in each of the possible situations $s\in\sits$; see the figure below.
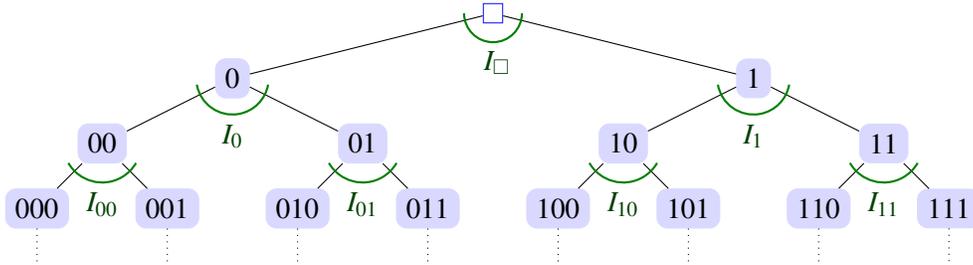
\begin{center}
\begin{tikzpicture}
\tikzstyle{level 1}=[sibling distance=18em]
\tikzstyle{level 2}=[sibling distance=9em]
\tikzstyle{level 3}=[sibling distance=4.5em]
\tikzstyle{level 4}=[level distance=2em]
\node[root] (root) {} [grow=down,level distance=5ex]
child {node[nonterminal] (a) {$0$}
  child {node[nonterminal] (aa) {$00$}
    child {node[nonterminal] (aaa) {$000$}
      child[black,dotted]}
    child {node[nonterminal] (aab) {$001$}
      child[black,dotted]}
  }
  child {node[nonterminal] (ab) {$01$}
    child {node[nonterminal] (aba) {$010$}
      child[black,dotted]}
    child {node[nonterminal] (abb) {$011$}
      child[black,dotted]}
  }
}
child {node[nonterminal] (b) {$1$}
  child {node[nonterminal] (ba) {$10$}
    child {node[nonterminal] (baa) {$100$}
      child[black,dotted]}
    child {node[nonterminal] (bab) {$101$}
      child[black,dotted]}
  }
  child {node[nonterminal] (bb) {$11$}
    child {node[nonterminal] (bba) {$110$}
      child[black,dotted]}
    child {node[nonterminal] (bbb) {$111$}
      child[black,dotted]}
  }
};
\draw[local,thick] (root) +(180:1em) arc (180:360:1em);
\draw[local,thick] (b) +(190:1.25em) arc (190:350:1.25em);
\draw[local,thick] (a) +(190:1.25em) arc (190:350:1.25em);
\draw[local,thick] (bb) +(210:1.35em) arc (210:330:1.35em);
\draw[local,thick] (ba) +(210:1.35em) arc (210:330:1.35em);
\draw[local,thick] (ab) +(210:1.35em) arc (210:330:1.35em);
\draw[local,thick] (aa) +(210:1.35em) arc (210:330:1.35em);
\path (root) +(275:1.7em) node[local] {$I_\init$};
\path (a) +(270:2em) node[local] {$I_{0}$};
\path (b) +(270:2em) node[local] {$I_{1}$};
\path (aa) +(270:2.1em) node[local] {$I_{00}$};
\path (bb) +(270:2.1em) node[local] {$I_{11}$};
\path (ba) +(270:2.1em) node[local] {$I_{10}$};
\path (ab) +(270:2.1em) node[local] {$I_{01}$};
\end{tikzpicture}

\end{center}
Indeed, we can use this idea to generalise the notion of a forecasting system \citep{vovk2010:randomness}.

\begin{definition}[Forecasting system]
A \emph{forecasting system} is a map $\frcstsystem\colon\sits\to\imprecisefrcsts$, that associates with any situation $s$ in the event tree a forecast\/ $\frcstsystem(s)\in\imprecisefrcsts$.
With any forecasting system $\frcstsystem$ we can associate two real-valued maps $\lfrcstsystem$ and\/ $\ufrcstsystem$ on $\sits$, defined by $\lfrcstsystem(s)\coloneqq\min\frcstsystem(s)$ and\/ $\ufrcstsystem(s)\coloneqq\max\frcstsystem(s)$ for all $s\in\sits$.
A forecasting system $\frcstsystem$ is called \emph{precise} if\/ $\lfrcstsystem=\ufrcstsystem$.
$\frcstsystems$ denotes the set $\smash{\imprecisefrcsts^{\sits}}$ of all forecasting systems.
\end{definition}
\noindent
Specifying such a forecasting system requires imagining in advance all moves that Reality 
could make, and devising in advance what forecasts to give in each imaginable situation $s$.
In the precise case, that is typically what one does when specifying a probability measure on the so-called \emph{sample space} $\pths$---the set $\pths$ of all paths.

Since in each situation $s$ the interval forecast $I_s=\frcstsystem(s)$ corresponds to a local lower expectation $\lex_{I_s}$, we can use the argumentation in our earlier papers \citep{cooman2007d,cooman2015:markovergodic} on stochastic processes to let the forecasting system $\frcstsystem$ turn the event tree into a so-called \emph{imprecise probability tree}, with an associated global lower expectation, and a corresponding notion of `(strictly) almost surely'.
In what follows, we briefly recall how to do this; for more context, we also refer to the seminal work by \citet{shafer2001}.

For any path $\pth\in\pths$, the initial sequence that consists of its first $n$ elements is a situation in $\outcomes^n$ that is denoted by $\pth^{n}$. 
Its $n$-th element belongs to $\outcomes$ and is denoted by $\pth_n$.
As a convention, we let its $0$-th element be the \emph{initial} situation $\pth^0=\pth_0=\init$. 
We write that $s\precedes t$, and say that the situation $s$ \emph{precedes} the situation $t$, when every path that goes through $t$ also goes through $s$---so $s$ is a precursor of $t$.

A \emph{process} $\process$ is a map defined on $\sits$. 
A \emph{real process} is a real-valued process: it associates a real number $\process(s)\in\reals$ with every situation $s\in\sits$.
With any real process $\process$, we can always associate a process $\Delta\process$, called the \emph{process difference}. 
For every situation $(\xvalto{n})$ with $n\in\naturalswithzero$, $\Delta\process(\xvalto{n})$ is a gamble on $\outcomes$ defined by $\Delta\process(\xvalto{n})(\xval{n+1})\coloneqq\process(\xvalto{n+1})-\process(\xvalto{n})$ for all $\xval{n+1}\in\outcomes$.
In the imprecise probability tree associated with a \emph{given} forecasting system $\frcstsystem$, a \emph{submartingale} $\submartin$ for $\frcstsystem$ is a real process such that $\lex_{\frcstsystem(\xvalto{n})}(\Delta\submartin(\xvalto{n}))\geq0$ for all $n\in\naturalswithzero$ and $(\xvalto{n})\in\outcomes^n$.
A real process $\supermartin$ is a \emph{supermartingale} for $\frcstsystem$ if $-\supermartin$ is a submartingale, meaning that $\uex_{\frcstsystem(\xvalto{n})}(\Delta\submartin(\xvalto{n}))\leq0$ for all $n\in\naturalswithzero$ and $(\xvalto{n})\in\outcomes^n$: all supermartingale differences have non-positive upper expectation, so supermartingales are real processes that Forecaster expects to decrease.
We denote the set of all submartingales for a given forecasting system $\frcstsystem$ by $\submartins[\frcstsystem]$---whether a real process is a submartingale depends of course on the forecasts in the situations.
Similarly, the set $\supermartins[\frcstsystem]\coloneqq-\submartins[\frcstsystem]$ is the set of all supermartingales for $\frcstsystem$.

It is clear from the discussion in Section~\ref{sec:single:forecast} that the supermartingales are effectively all the possible capital processes $\capitalsymbol$ for a Sceptic who starts with an initial capital $\capitalsymbol(\init)$, and in each possible subsequent situation $s$ selects a gamble $f_s=\Delta\capitalsymbol(s)$ that is available there because Forecaster specifies the interval forecast $I_s=\frcstsystem(s)$ and because $\uex_{I_s}(f_s)=\uex_{\frcstsystem(s)}(\Delta\capitalsymbol(s))\leq0$. 
If Reality chooses outcomes $s=(\xvalto{n})$, then Sceptic ends up with capital $\capitalsymbol(\xvalto{n})=\capitalsymbol(\init)+\sum_{k=0}^{n-1}\Delta\capitalsymbol(\xvalto{k})(\xval{k+1})$.
A \emph{non-negative} supermartingale $\supermartin$ is non-negative in all situations, which corresponds to Sceptic never borrowing any money. 
We call \emph{test supermartingale} any non-negative supermartingale $\supermartin$ that starts with unit capital $\supermartin(\init)=1$.
We collect all test supermartingales for $\frcstsystem$ in the set $\testsupermartins[\frcstsystem]$. 

In the context of probability trees, we call \emph{variable} any function defined on the sample space~$\pths$.
When this variable is real-valued and bounded, we call it a \emph{gamble} on $\pths$.
An \emph{event} $A$ in this context is a subset of $\pths$, and its indicator $\ind{A}$ is a gamble on $\pths$ assuming the value $1$ on $A$ and $0$ elsewhere.
The following expressions define lower and upper expectations on such gambles $g$ on $\pths$:
\begin{align}
\ljoint[\frcstsystem](g)
\coloneqq&\sup\cset[\Big]{\submartin(\init)}
{\submartin\in\submartins[\frcstsystem]
\text{ and }
\limsup_{n\to+\infty}\submartin(\pth^n)\leq g(\pth)
\text{ for all $\pth\in\pths$}}
\label{eq:tree:lower:expectation}\\
\ujoint[\frcstsystem](g)
\coloneqq&\inf\cset[\Big]{\supermartin(\init)}
{\supermartin\in\supermartins[\frcstsystem]
\text{ and }
\liminf_{n\to+\infty}\supermartin(\pth^n)\geq g(\pth)
\text{ for all $\pth\in\pths$}}
=-\ljoint[\frcstsystem](g).
\label{eq:tree:upper:expectation}
\end{align}
They satisfy coherence properties similar to those in Proposition~\ref{prop:properties:of:expectations}.
We refer to extensive discussions elsewhere \citep{cooman2015:markovergodic,shafer2001} about why these expressions are interesting and useful.
For our present purposes, it may suffice to mention that for precise forecasts, they lead to  models that coincide with the ones found in measure-theoretic probability theory \citep[Chapter~8]{shafer2001}. 
\emph{In particular, when all $I_s=\set{\nicefrac{1}{2}}$, they coincide with the usual uniform (Lebesgue) expectations on measurable gambles.}

We call an event $A\subseteq\pths$ \emph{null} if $\overline{P}^{\frcstsystem}(A)\coloneqq\ujoint[\frcstsystem](\ind{A})=0$, or equivalently $\underline{P}^{\frcstsystem}(A^c)\coloneqq\ljoint[\frcstsystem](\ind{A^c})=1$, and \emph{strictly null} if there is some test supermartingale $\test\in\testsupermartins[\frcstsystem]$ that converges to $+\infty$ on $A$, meaning that $\lim_{n\to+\infty}\test(\pth^n)=+\infty$ for all $\pth\in A$.
Any strictly null event is null, but null events need not be strictly null~\citep{vovk2014:itip,cooman2015:markovergodic}.
Because it is easily checked that $\overline{P}^{\frcstsystem}(\emptyset)=\underline{P}^{\frcstsystem}(\emptyset)=0$
, the complement $A^c$ of a (strictly) null event $A$ is never empty.
As usual, any property that holds, except perhaps on a (strictly) null event, is said to hold (strictly) \emph{almost surely}.

\section{Basic computability notions}\label{sec:computability}
We recall a few notions and results from computability theory that are relevant to the discussion.
For a much more extensive treatment, we refer for instance to the books by \citet{pourel1989} and \citet{li1993}.

A \emph{computable} function $\phi\colon\naturalswithzero\to\naturalswithzero$ is a function that can be computed by a Turing machine.
All notions of computability that we will need, build on this basic notion.
It is clear that it in this definition, we can replace any of the $\naturalswithzero$ with any other countable set.

We start with the definition of a computable real number.
We call a sequence of rational numbers $r_n$ \emph{computable} if there are three computable functions $a,b,\sigma$ from $\naturalswithzero$ to $\naturalswithzero$ such that $b(n)>0$ and $r_n=(-1)^{\sigma(n)}\frac{a(n)}{b(n)}$ for all $n\in\naturalswithzero$, and we say that it \emph{converges effectively} to a real number $x$ if there is some computable function $e\colon\naturalswithzero\to\naturalswithzero$ such that $n\geq e(N)\then\abs{r_n-x}\leq2^{-N}$  for all $n,N\in\naturalswithzero$.
A real number is then called \emph{computable} if there is a computable sequence of rational numbers that converges effectively to it.
Of course, every rational number is a computable real.

We also need a notion of computable real processes, or in other words, computable real-valued maps $\process\colon\sits\to\reals$ defined on the set $\sits$ of all situations.
Because there is an obvious computable bijection between $\naturalswithzero$ and $\sits$, whose inverse is also computable, we can in fact identify real processes and real sequences, and simply import, {\itshape mutatis mutandis}, the definitions for computable real sequences common in the literature \citep[Chapter~0]{li1993}.
Indeed, we call a net of rational numbers $r_{s,n}$ \emph{computable} if there are three computable functions $a,b,s$ from $\sits\times\naturalswithzero$ to $\naturalswithzero$ such that $b(s,n)>0$ and $r_{s,n}=(-1)^{\sigma(s,n)}\frac{a(s,n)}{b(s,n)}$ for all $s\in\sits$ and $n\in\naturalswithzero$.
We call a real process $\process\colon\sits\to\reals$ \emph{computable} if there is a computable net of rational numbers $r_{s,n}$ and a computable function $e\colon\sits\times\naturalswithzero\to\naturalswithzero$ such that $n\geq e(s,N)\then\abs{r_{s,n}-\process(s)}\leq2^{-N}$ for all $s\in\sits$ and $n,N\in\naturalswithzero$.
Obviously, it follows from this definition that in particular $\process(t)$ is a computable real number for any $t\in\sits$: fix $s=t$ and consider the sequence $r_{t,n}$ that converges to $\process(s)$ as $n\to+\infty$.
Also, a constant real process is computable if and only if its constant value is.

The following definitions are now obvious. 
A gamble $f$ on $\outcomes$ is called \emph{computable} if both its values $f(0)$ and $f(1)$ are computable real numbers.
An interval forecast $I=\pinterval\in\imprecisefrcsts$ is called \emph{computable} if both its lower bound $\lp$ and upper bound $\up$ are computable real numbers.
A forecasting system $\frcstsystem$ is called \emph{computable} if the associated real processes $\lfrcstsystem$ and $\ufrcstsystem$ are.

\section{Random sequences in an imprecise probability tree}\label{sec:randomness}
We will now associate a notion of randomness with a forecasting system $\frcstsystem$---or in other words, with an imprecise probability tree.
In what follows, we will often consider computable test supermartingales.
These computable test supermartingales for a forecasting system are countable in number, because the computable processes are \citep{li1993,vovk2010:randomness}.

\begin{definition}[Computable randomness]\label{def:randomness}
Consider any forecasting system $\frcstsystem\colon\sits\to\imprecisefrcsts$.
We call an outcome sequence $\pth$ \emph{computably random for $\frcstsystem$} if all computable test supermartingales $\test$ remain bounded above on $\pth$, meaning that there is some $B\in\reals$ such that $\test(\pth^n)\leq B$ for all $n\in\naturals$, or equivalently, that $\sup_{n\in\naturals}\test(\pth^n)<+\infty$.
We then also say that the forecasting system $\frcstsystem$ \emph{makes $\pth$ computably random}.
We denote by $\random[\pth]{\mathrm{C}}\coloneqq\cset{\frcstsystem\in\frcstsystems}{\text{$\pth$ is computably random for $\frcstsystem$}}$ the set of all forecasting systems for which the outcome sequence $\pth$ is computably random.
\end{definition}
\noindent
Computable randomness of an outcome sequence means that there is no computable strategy that starts with capital $1$ and avoids borrowing, and allows Sceptic to increase his capital without bounds by exploiting the bets on these outcomes that are made available to him by Forecaster's specification of the forecasting system $\frcstsystem$. 
When the forecasting system $\frcstsystem$ is precise and computable, our notion of computable randomness reduces to the classical notion of computable randomness \citep{ambosspies2000,bienvenu2009:randomness}.

The (computable) \emph{vacuous} forecasting system $\frcstsystem_\mathrm{v}$ assigns the vacuous forecast $\frcstsystem_\mathrm{v}(s)\coloneqq\frcsts$ to all situations $s\in\sits$.
The following proposition implies that no $\random[\pth]{\mathrm{C}}$ is empty.

\begin{proposition}\label{prop:vacuous}
All paths are computably random for the vacuous forecasting system: $\frcstsystem_\mathrm{v}\in\random[\pth]{\mathrm{C}}$ for all $\pth\in\pths$.
\end{proposition}
\noindent
More conservative (or imprecise) forecasting systems have more computably random sequences.
\begin{proposition}\label{prop:nestedfrcstsystems}
Let $\pth$ be computably random for a forecasting system $\frcstsystem$. 
Then $\pth$ is also computably random for any forecasting system $\frcstsystem^*$ such that $\frcstsystem\subseteq\frcstsystem^*$, meaning that $\frcstsystem(s)\subseteq\frcstsystem^*(s)$ for all $s\in\sits$.
\end{proposition}

\section{Consistency results}
\label{sec:consistency}
We first show that any Forecaster who specifies a forecasting system is consistent in the sense that he believes himself to be \emph{well calibrated}: in the imprecise probability tree generated by his own forecasts, (strictly) almost all paths will be computably random, so he is sure that Sceptic will not be able to become infinitely rich at his expense, by exploiting his---Forecaster's---forecasts.
This also generalises the arguments and conclusions in a paper by \citet{dawid1982:well:calibrated:bayesian}.

\begin{theorem}\label{thm:consistency}
Consider any forecasting system $\frcstsystem\colon\sits\to\imprecisefrcsts$.
Then (strictly) almost all outcome sequences are computably random for $\frcstsystem$ in the imprecise probability tree that corresponds to~$\frcstsystem$.
\end{theorem}
\noindent
This result is quite powerful, and it guarantees in particular that:

\begin{corollary}\label{cor:consistency}
For any sequence of interval forecasts $(I_1,\dots,I_n,\dots)$ there is a forecasting system given by $\frcstsystem(x_1,\dots,x_n)\coloneqq I_{n+1}$
 for all $(x_1,\dots,x_n)\in\outcomes^n$ and all $n\in\naturalswithzero$, and associated imprecise probability tree such that (strictly) almost all---and therefore definitely at least one---outcome sequences are computably random for $\frcstsystem$ in the associated imprecise probability tree. 
\end{corollary}
\noindent
The following weaker consistency result deals with limits (inferior and superior) of relative frequencies, taken with respect to a so-called \emph{selection process} $\selection\colon\sits\to\outcomes$.
It is a counterpart in our more general context of the notions of \emph{computable stochasticity} or \emph{Church randomness} in the precise case with $I=\set{\nicefrac{1}{2}}$ \citep{ambosspies2000}.

\begin{theorem}[Church randomness]\label{thm:well:calibrated:general}
Let $\frcstsystem\colon\sits\to\imprecisefrcsts$ be any \emph{computable} forecasting system, let $\pth=(\xvalto{n},\dots)\in\pths$ be any outcome sequence that is computably random for $\frcstsystem$, and let $f$ be any \emph{computable} gamble on $\outcomes$.
If\/ $\selection\colon\sits\to\outcomes$ is any computable selection process such that $\sum_{k=0}^n\selection(\xvalto{k})\to+\infty$, then also
\begin{equation*}
\liminf_{n\to+\infty}
\dfrac{\sum_{k=0}^{n-1}\selection(\xvalto{k})\sqgroup[\big]{f(x_{k+1})-\lex_{\frcstsystem(\xvalto{k})}(f)}}
{\sum_{k=0}^{n-1}\selection(\xvalto{k})} 
\geq0.
\end{equation*}
\end{theorem}

\section{Constant interval forecasts}
\label{sec:constantintervalforecasts}
We now introduce a significant simplification.
For any interval $I\in\imprecisefrcsts$, we let $\constantfrcst{I}$ be the corresponding \emph{stationary forecasting system} that assigns the same interval forecast $I$ to all nodes: $\constantfrcst{I}(s)\coloneqq I$ for all $s\in\sits$.
In this way, with any outcome sequence $\pth$, we can associate the collection of all interval forecasts for which the corresponding stationary forecasting system makes $\pth$ computably random:
\begin{equation*}
\constantrandom{\mathrm{C}}
\coloneqq\cset{I\in\imprecisefrcsts}{\constantfrcst{I}\in\random{\mathrm{C}}}
=\cset{I\in\imprecisefrcsts}{\constantfrcst{I}\text{ makes $\pth$ computably random}}.
\end{equation*}
As an immediate consequence of Propositions~\ref{prop:vacuous} and~\ref{prop:nestedfrcstsystems}, we find that this set of intervals is non-empty and increasing.

\begin{proposition}[Non-emptiness]\label{prop:constantcalibrated:top}
For all $\pth\in\pths$, $\frcsts\in\constantrandom{\mathrm{C}}$, so any sequence of outcomes $\pth$ has at least one stationary forecast that makes it computably random: $\constantrandom{\mathrm{C}}\neq\emptyset$.
\end{proposition}

\begin{proposition}[Increasingness]\label{prop:constantcalibrated:increasing}
Consider any $\pth\in\pths$ and any $I,J\in\imprecisefrcsts$.
If $I\in\constantrandom{\mathrm{C}}$ and $I\subseteq J$, then also $J\in\constantrandom{\mathrm{C}}$.
\end{proposition}
\noindent
Theorem~\ref{thm:well:calibrated:general} implies the following property. 
However, quite remarkably, and seemingly in contrast with Theorem~\ref{thm:well:calibrated:general}, this result does not require any computability assumptions on the (stationary) forecasts.

\begin{corollary}[Church randomness]\label{cor:well:calibrated:constant}
Consider any outcome sequence $\pth=(x_1,\dots,x_n,\dots)$ in $\pths$ and any stationary interval forecast $I=\pinterval\in\constantrandom{\mathrm{C}}$ that makes $\pth$ computably random.
Then for any computable selection process $\selection\colon\sits\to\outcomes$ such that $\sum_{k=0}^n\selection(x_1,\dots,x_k)\to+\infty$:
\begin{equation*}
\lp
\leq\liminf_{n\to+\infty}
\frac{\sum_{k=0}^{n-1}\selection(x_1,\dots,x_k)x_{k+1}}
{\sum_{k=0}^{n-1}\selection(x_1,\dots,x_k)}
\leq\limsup_{n\to+\infty}
\frac{\sum_{k=0}^{n-1}\selection(x_1,\dots,x_k)x_{k+1}}
{\sum_{k=0}^{n-1}\selection(x_1,\dots,x_k)}
\leq\up.
\end{equation*}
\end{corollary}
\noindent
The following proposition can of course be straightforwardly extended to any finite number of interval forecasts, and guarantees, together with Proposition~\ref{prop:constantcalibrated:increasing}, that $\constantrandom{\mathrm{C}}$ is a \emph{set filter}.

\begin{proposition}\label{prop:constantcalibrated:nonempty:intersection:inside:ML}
For any $\pth\in\pths$ and any two interval forecasts $I$ and $J$: if $I\in\constantrandom{\mathrm{C}}$ and $J\in\constantrandom{\mathrm{C}}$ then $I\cap J\neq\emptyset$, and $I\cap J\in\constantrandom{\mathrm{C}}$.
\end{proposition}
\noindent
This result also tells us that the collection $\constantrandom{\mathrm{C}}$ of closed subsets of the compact set $\frcsts$ has the finite intersection property, and its intersection is therefore a non-empty closed interval: $\bigcap\constantrandom{\mathrm{C}}=\sqgroup{\lowersmallestrandom[\pth]{\mathrm{C}},\uppersmallestrandom[\pth]{\mathrm{C}}}$.
Propositions~\ref{prop:constantcalibrated:increasing} and~\ref{prop:constantcalibrated:nonempty:intersection:inside:ML} guarantee that all intervals $\sqgroup{\lowersmallestrandom[\pth]{\mathrm{C}}-\epsilon_1,\uppersmallestrandom[\pth]{\mathrm{C}}+\epsilon_2}$ in $\imprecisefrcsts$ with $\epsilon_1,\epsilon_2>0$ belong to $\constantrandom{\mathrm{C}}$.
But we will see in the next section that this does not generally hold for $\epsilon_1=0$ and/or $\epsilon_2=0$.
For this reason, we now define the following two subsets of $\frcsts$:
\begin{equation*}
\lowerconstantrandom[\pth]{\mathrm{C}}
\coloneqq\cset{\min I}{I\in\constantrandom{\mathrm{C}}}
\text{ and }
\upperconstantrandom[\pth]{\mathrm{C}}
\coloneqq\cset{\max I}{I\in\constantrandom{\mathrm{C}}}.
\end{equation*}
Then Proposition~\ref{prop:constantcalibrated:increasing} guarantees that $\lowerconstantrandom[\pth]{\mathrm{C}}$ is a decreasing set, and that $\upperconstantrandom[\pth]{\mathrm{C}}$ is increasing.
They are therefore both subintervals of $\frcsts$.
Obviously, $\lowersmallestrandom[\pth]{\mathrm{C}}=\sup\lowerconstantrandom[\pth]{\mathrm{C}}$ and $\uppersmallestrandom[\pth]{\mathrm{C}}=\inf\upperconstantrandom[\pth]{\mathrm{C}}$.
On the one hand clearly $\lowerconstantrandom[\pth]{\mathrm{C}}=[0,\lowersmallestrandom[\pth]{\mathrm{C}})$ or $\lowerconstantrandom[\pth]{\mathrm{C}}=[0,\lowersmallestrandom[\pth]{\mathrm{C}}]$, and on the other hand $\upperconstantrandom[\pth]{\mathrm{C}}=(\uppersmallestrandom[\pth]{\mathrm{C}},1]$ or $\upperconstantrandom[\pth]{\mathrm{C}}=[\uppersmallestrandom[\pth]{\mathrm{C}},1]$. 
Proposition~\ref{prop:constantcalibrated:nonempty:intersection:inside:ML} easily allows us to give the following simple description of the set $\constantrandom{\mathrm{C}}$ in terms of $\lowerconstantrandom[\pth]{\mathrm{C}}$ and $\upperconstantrandom[\pth]{\mathrm{C}}$:
\begin{equation*}
I\in\constantrandom[\pth]{\mathrm{C}}
\ifandonlyif
\group[\Big]{%
\min I\in\lowerconstantrandom[\pth]{\mathrm{C}}
\text{ and }
\max I\in\upperconstantrandom[\pth]{\mathrm{C}}
}.
\end{equation*}
A trivial example is given by:
\begin{proposition}\label{prop:constantcalibrated:computable:sequence}
If the sequence $\pth$ is computable with infinitely many zeroes and ones, then $\constantrandom{\mathrm{C}}=\set{[0,1]}$, and therefore $\lowerconstantrandom[\pth]{\mathrm{C}}=\set{0}$, $\upperconstantrandom[\pth]{\mathrm{C}}=\set{1}$, $\lowersmallestrandom[\pth]{\mathrm{C}}=0$ and $\uppersmallestrandom[\pth]{\mathrm{C}}=1$. 
\end{proposition}
\noindent
At the other extreme, there are the sequences $\pth$ that are computably random for some \emph{precise} stationary forecasting system $\prcsfrcstsystem$, with $p\in\frcsts$.
They are amongst the random sequences that have received most attention in the literature, thus far.
For any such sequence, $\constantrandom[\pth]{\mathrm{C}}=\cset{I\in\imprecisefrcsts}{p\in I}$, $\lowerconstantrandom[\pth]{\mathrm{C}}=[0,p]$ and $\upperconstantrandom[\pth]{\mathrm{C}}=[p,0]$, and therefore also $\lowersmallestrandom[\pth]{\mathrm{C}}=\uppersmallestrandom[\pth]{\mathrm{C}}=p$.

We show in the next section that, in between these extremes of total imprecision and maximal precision, there lies a---to the best of our knowledge---previously uncharted realm of sequences, with similar (and even in some sense `larger') unpredictability than the ones traditionally called `computably random', for which $\lowerconstantrandom[\pth]{\mathrm{C}}$ and $\upperconstantrandom[\pth]{\mathrm{C}}$ need not always be closed, and more importantly, for which $0<\lowersmallestrandom[\pth]{\mathrm{C}}<\uppersmallestrandom[\pth]{\mathrm{C}}<1$.
This is what we mean when we claim that `computable randomness is inherently imprecise'.

\section{Randomness is inherently imprecise}\label{sec:inherently}
Our work on imprecise Markov chains \citep{cooman2015:markovergodic} has taught us that in some cases, we can very efficiently compute tight bounds on expectations in non-stationary precise Markov chains, by replacing them with their stationary imprecise versions.
Similarly, in statistical modelling, when learning from data sampled from a distribution with a varying (non-stationary) parameter, it seems hard to estimate the exact time sequence of its values. 
But we may be more successful in learning about its (stationary) interval \emph{range}.
This idea was also considered earlier by \citet{fierens2009:frequentist}, when they argued for a frequentist interpretation of imprecise probability models based on non-stationarity.

In this section, we exploit this idea, by showing that randomness associated with non-stationary precise forecasting systems can be captured by a stationary forecasting system, which must then be less precise: we gain simplicity of representation, but pay for it by losing precision.

We begin with a simple example.
Consider any $p$ and $q$ in $\frcsts$ with $p\leq q$, and any outcome sequence $\pth=(x_1,\dots,x_n,\dots)$ that is computably random for the forecasting system $\frcstsystem_{p,q}$ that is defined by
\begin{equation*}
\frcstsystem_{p,q}(z_1,\dots,z_n)
\coloneqq
\begin{cases}
p &\text{if $n$ is odd}\\
q &\text{if $n$ is even}
\end{cases}
\quad\text{for all $(z_1,\dots,z_n)\in\sits$.}\vspace{2pt}
\end{equation*}
We know from Corollary~\ref{cor:consistency} that there is at least one such outcome sequence.
It turns out that the stationary forecasting systems that make such $\pth$ computably random have a simple characterisation:

\begin{proposition}\label{prop:first:example}
Consider any $\pth$ that is computably random for the forecasting system $\frcstsystem_{p,q}$.
Then for all $I\in\imprecisefrcsts$, $I\in\constantrandom{\mathrm{C}}\ifandonlyif\sqgroup{p,q}\subseteq I$.
\end{proposition}
\noindent Its proof relies on a very simple argument involving Corollary~\ref{cor:well:calibrated:constant}.
This result implies in particular also that $\lowerconstantrandom[\pth]{\mathrm{C}}=\sqgroup{0,p}$, $\upperconstantrandom[\pth]{\mathrm{C}}=\sqgroup{q,1}$, $\lowersmallestrandom[\pth]{\mathrm{C}}=p$ and $\uppersmallestrandom[\pth]{\mathrm{C}}=q$.

Next, we turn to a more complicated example, where we look at sequences that are `nearly' computably random for the stationary precise forecast $\nicefrac{1}{2}$, but not quite.
This example was inspired by the ideas involving Hellinger-like divergences in a beautiful paper by \citet{vovk2009:merging}.

Consider the following sequence $\{p_n\}_{n\in\naturals}$ of precise forecasts:
\begin{equation*}
p_{n}\coloneqq\frac{1}{2}+(-1)^{n}\delta_{n}
\text{, with }
\delta_n\coloneqq e^{-\frac{1}{n+1}}\sqrt{e^{\frac{1}{n+1}}-1}
\text{ for all $n\in\naturals$},
\end{equation*}
converging to $\nicefrac{1}{2}$.
Observe that the sequence $\delta_n$ is decreasing towards its limit $0$ and that $\delta_n\in(0,\nicefrac{1}{2})$ and $p_n\in(0,1)$, for all $n\in\naturals$.
Now consider any outcome sequence $\pth=(x_1,\dots,x_n,\dots)$ that is computably random for the precise forecasting system $\xmplfrcstsystem$ that is defined by
\begin{equation*}
\xmplfrcstsystem(z_1,\dots,z_{n-1})\coloneqq p_{n}
\text{ for all $n\in\naturals$ and $(z_1,\dots,z_{n-1})\in\sits$.}
\end{equation*}
We know from Corollary~\ref{cor:consistency} that there is at least one such outcome sequence.
It turns out that the stationary forecasting systems that make such $\pth$ computably random have a simple characterisation:

\begin{proposition}\label{prop:second:example}
Consider any $\pth$ that is computably random for the forecasting system $\frcstsystem_{\sim\nicefrac{1}{2}}$.
Then for all $I\in\imprecisefrcsts$, $I\in\constantrandom{\mathrm{C}}$ if and only if\/ $\min I<\nicefrac{1}{2}$ and\/ $\max I>\nicefrac{1}{2}$.
\end{proposition}
\noindent This result implies in particular that $\lowerconstantrandom[\pth]{\mathrm{C}}=[0,\nicefrac{1}{2})$, $\upperconstantrandom[\pth]{\mathrm{C}}=(\nicefrac{1}{2},1]$ and $\lowersmallestrandom[\pth]{\mathrm{C}}=\uppersmallestrandom[\pth]{\mathrm{C}}=\nicefrac{1}{2}$.

\section{Conclusion}
Even with the limited number of examples we have been able to examine in this paper, it becomes apparent that incorporating imprecision in the study of randomness allows for much more mathematical structure to arise, which we would argue lets us better understand and place the existing results in the precise limit.

In our argumentation that `randomness is inherently imprecise', we are well aware that we are restricting ourselves to stationary forecasts. 
Our examples in Section~\ref{sec:inherently} all involve sequences that are computably random for a precise non-stationary forecasting system, but no longer computably random for any stationary precise variant. 
To make our claim irrefutable, we would have to show that there are sequences that are computably random for forecasting systems more precise than the vacuous one, but not for any (computable) precise forecasting system. 
Or in other words, that there is `randomness' or `unpredictability' that cannot be `explained' by any non-stationary (computable) precise forecasting system.
We will keep this challenge foremost in our minds.
Nevertheless, the examples in Section~\ref{sec:inherently} do indicate that it is in some ways possible to replace an `explanation' by a complex non-stationary precise forecasting model by a(n infinite filter of) more imprecise stationary one(s).

This work may seem promising, but we are well aware that it is only a humble beginning.
We see many extensions in many directions.
First of all, we want to find out if our approach can also be used to find interval versions of \emph{Martin-L\"of} and \emph{Schnorr randomness} \citep{ambosspies2000,bienvenu2009:randomness} with similarly interesting properties and conclusions.
Secondly, our preliminary exploration suggests that it will be possible to formulate equivalent randomness definitions in terms of \emph{randomness tests}, rather than supermartingales, but this needs to be checked in much more detail.
Thirdly, the approach we follow here is not prequential: we assume that our Forecaster specifies an entire forecasting system $\frcstsystem$, or in other words an interval forecast in all possible situations $(\xvalto{n})$, rather than only interval forecasts in those situations $\zvalto{n}$ of the sequence $\pth=(\zvalto{n},\dots)$ whose potential randomness we are considering.
The \emph{prequential approach}, which we eventually will want to come to, looks at the randomness of a sequence of interval forecasts and outcomes $(I_1,z_1,I_2,z_2,\dots,I_n,z_n,\dots)$, where each $I_k$ is an interval forecast for the as yet unknown $\randout[k]$, which is afterwards revealed to be $z_k$, without the need of specifying forecasts in situations that are never reached; see the paper by \citet{vovk2010:randomness} for an account of how this works for precise forecasts. 
Fourthly, we need to connect our work with earlier approaches to associating imprecision with randomness\ \citep{walley1982a,fierens2009:frequentist,fierens2009:chaotic,gorban2016:statisticalstability}.
And finally, and perhaps most importantly, we believe this research could be a very early starting point for an approach to statistics that takes imprecise or set-valued parameters more seriously, when learning from finite amounts of data.

\acks{This research started with discussions between Gert and Philip Dawid about what prequential interval forecasting would look like, during a joint stay at Durham University in late 2014. 
Gert, and Jasper who joined in late 2015, wrote an early prequential version of the present paper during a joint research visit to the Universities of Strathclyde and Durham in May 2016, trying to extend the results by Volodya Vovk~\citep{vovk1987:randomness,vovk2009:merging,vovk2010:randomness} to make them allow for interval forecasts.
In an email exchange, Volodya pointed out a number of difficulties with our approach, which we were able to resolve by letting go of its prequential emphasis, at least for the time being.
This was done during research visits of Gert to Jasper at IDSIA in late 2016 and early 2017.

We are grateful to Philip and Volodya for their inspiring and helpful comments and guidance.
Gert's research and travel were partly funded through project number G012512N of the Research Foundation -- Flanders (FWO), Jasper is a Postdoctoral Fellow of the FWO and wishes to acknowledge its financial support.}

\vskip 0.2in

\newpage
\appendix
\section{Proofs and additional material}
\label{app:proofs}
In this Appendix, we have gathered all proofs, and all additional material necessary for understanding the argumentation in these proofs.



\subsection{Additional material for Section~\ref{sec:forecasting:systems}}
\label{app:forecasting:systems}
In the interest of understanding the proofs, we need to pay attention to a particular way of constructing test supermartingales.
We define a \emph{multiplier process} as a map $\multprocess$ from $\sits$ to \emph{non-negative} gambles on $\outcomes$
Given such a multiplier process $\multprocess$, we can construct a non-negative real process $\mint$ by the recursion equation $\mint(sx)\coloneqq\mint(s)\multprocess(s)(x)$ for all $s\in\sits$ and $x\in\outcomes$, with $\mint(\init)\coloneqq1$.
Any multiplier process $\multprocess$ that satisfies the additional condition that $\uex_{\frcstsystem(s)}(\multprocess(s))\leq1$ for all $s\in\sits$, is called a \emph{supermartingale multiplier} for the forecasting system $\frcstsystem$.
It is easy to see that the non-negative real process $\mint$ is then a test supermartingale for $\frcstsystem$: it suffices to check that
\begin{equation}\label{eq:supermartingale:multiplier:differences}
\adddelta\mint(s)
=\mint(s)[\multprocess(s)-1],
\end{equation} 
and therefore $\uex_{\frcstsystem(s)}(\adddelta\mint(s))=\mint(s)[\uex_{\frcstsystem(s)}(\multprocess(s))-1]\leq0$, due to the coherence properties~\ref{axiom:coherence:homogeneity} and~\ref{axiom:coherence:constantadditivity} of upper expectation operators.

\subsection{Additional material for Section~\ref{sec:computability}}
\label{app:computability}
We give a brief survey of those basic notions and results from computability theory that are relevant to the proofs in this appendix.
For a much more extensive discussion, we refer, for instance to the books by \citet{pourel1989} and \citet{li1993}. 
The discussion in Section~\ref{sec:computability} is at times similar---and even identical---but is overall more limited in scope, as it only deals with aspects that are relevant to the main text.

A \emph{computable} function $\phi\colon\naturalswithzero\to\naturalswithzero$ is a function that can be computed by a Turing machine.
All further notions of computability that we will need are based on this basic notion.
It is clear that it in this definition, we can replace any of the $\naturalswithzero$ with any other countable set.

We start with the definition of a computable real number.
We call a sequence of rational numbers $r_n$ \emph{computable} if there are three computable functions $a,b,\sigma$ from $\naturalswithzero$ to $\naturalswithzero$ such that
\begin{equation*}
b(n)>0
\text{ and }
r_n=(-1)^{\sigma(n)}\frac{a(n)}{b(n)}
\text{ for all $n\in\naturalswithzero$},
\end{equation*}
and we say that it \emph{converges effectively} to a real number $x$ if there is some computable function $e\colon\naturalswithzero\to\naturalswithzero$ such that
\begin{equation*}
n\geq e(N)\then\abs{r_n-x}\leq2^{-N}
\text{ for all $n,N\in\naturalswithzero$}.
\end{equation*}
A real number is then called \emph{computable} if there is a computable sequence of rational numbers that converges effectively to it.
Of course, every rational number is a computable real.

We also need a notion of computable real processes, or in other words, computable real-valued maps $\process\colon\sits\to\reals$ defined on the set $\sits$ of all situations.
Because there is an obvious computable bijection between $\naturalswithzero$ and $\sits$, whose inverse is also computable, we can in fact identify real processes and real sequences, and simply import, {\itshape mutatis mutandis}, the definitions for computable real sequences common in the literature \citep[Chapter~0]{li1993}.
Indeed, we call a net of rational numbers $r_{s,n}$ \emph{computable} if there are three computable functions $a,b,s$ from $\sits\times\naturalswithzero$ to $\naturalswithzero$ such that
\begin{equation*}
b(s,n)>0
\text{ and }
r_{s,n}=(-1)^{\sigma(s,n)}\frac{a(s,n)}{b(s,n)}
\text{ for all $s\in\sits$ and $n\in\naturalswithzero$}.
\end{equation*}
We call a real process $\process\colon\sits\to\reals$ \emph{computable} if there is a computable net of rational numbers $r_{s,n}$ and a computable function $e\colon\sits\times\naturalswithzero\to\naturalswithzero$ such that
\begin{equation*}
n\geq e(s,N)\then\abs{r_{s,n}-\process(s)}\leq2^{-N}
\text{ for all $s\in\sits$ and $n,N\in\naturalswithzero$}. 
\end{equation*} 
Again, there is no problem with the notions `computable net of rational numbers' or `computable function' in this definition, because we can identify $\sits\times\naturalswithzero$ with $\naturalswithzero$ through a computable bijection whose inverse is also computable.
Obviously, it follows from this definition that in particular $\process(t)$ is a computable real number for any $t\in\sits$: fix $s=t$ and consider the sequence $r_{t,n}$ that converges to $\process(s)$ as $n\to+\infty$.
Also, a constant real process is computable if and only if its constant value is.

We recall the following standard results \citep[Chapter~0]{li1993}.

\begin{proposition}\label{prop:computable:simplified}
A real process $\process$ is computable if and only if there is a computable net of rational numbers $r_{s,n}$ such that $\abs{r_{s,n}-\process(s)}\leq2^{-n}$ for all $s\in\sits$ and $n\in\naturalswithzero$.
\end{proposition}

\begin{proof}{\bf of Proposition~\ref{prop:computable:simplified}\quad}
We give the proof for the sake of completeness.
The `if' part is immediate, so we proceed to the `only if' part.
That $\process$ is computable means that there is some computable net of rational numbers $r'_{s,n}$ and a computable function $e\colon\sits\times\naturalswithzero\to\naturalswithzero$ such that $n\geq e(s,N)$ implies $\abs{r'_{s,n}-F(s)}\leq2^{-N}$ for all $s\in\sits$ and $N\in\naturalswithzero$.
The net of rational numbers $r_{s,n}\coloneqq r'_{s,e(s,n)}$ is computable because the function $e$ is computable, and satisfies $\abs{r_{s,n}-\process(s)}=\abs{r'_{s,e(s,n)}-\process(s)}\leq2^{-n}$.
\end{proof}\vspace{-20pt}

\begin{proposition}\label{prop:computable:closure}
Consider any computable net $x_{s,n}$ of real numbers and any real process $\process$ for which there is some computable function $e\colon\sits\times\naturalswithzero\to\naturalswithzero$ such that
\begin{equation*}
n\geq e(s,N)\then\abs{x_{s,n}-\process(s)}\leq2^{-N}
\text{ for all $s\in\sits$ and $N\in\naturalswithzero$}. 
\end{equation*} 
Then $\process$ is computable.
\end{proposition}
\noindent
Also, if $\process$ and $\processtoo$ are computable real processes, then so are $\process+\processtoo$, $\process\processtoo$, $\process/\processtoo$ (provided that $\processtoo(s)\neq0$ for all $s\in\sits$), $\max\set{\process,\processtoo}$, $\min\set{\process,\processtoo}$, $\exp(\process)$, $\ln\process$ (provided that $\process(s)>0$ for all $s\in\sits$), and $F^{\frac{1}{m}}$ (provided that $\process(s)\geq0$ for all $s\in\sits$) for all $m\in\naturals$ \citep[Chapter~0]{li1993}.

We also require the notion of a {\scomp} real processes.
A real process $\process$ is \emph{\lscomp} if it can be approximated from below by a computable net of rational numbers, meaning that there is a computable net of rational numbers $r_{s,n}$ such that
\begin{enumerate}[label=\upshape(\roman*),leftmargin=*,noitemsep,topsep=0pt]
\item $r_{s,n+1}\geq r_{s,n}$ for all $s\in\sits$ and $n\in\naturalswithzero$;
\item $\process(s)=\lim_{n\to+\infty}r_{s,n}$ for all $s\in\sits$.
\end{enumerate}
$\process$ is \emph{\uscomp} if $-\process$ is {\lscomp}.
The following result is standard, but its proof is illustrative.

\begin{proposition}\label{prop:computable:upper:lower}
A process $\process$ is computable if and only is it is both lower and upper {\scomp}.
\end{proposition}

\begin{proof}{\bf of Proposition~\ref{prop:computable:upper:lower}\quad}
We begin with the `if' part.
Assume that $\process$ is both lower and upper {\scomp}.
This implies that there are two computable nets of rational numbers $\underline{r}_{s,n}$ and $\overline{r}_{s,n}$ such that $\underline{r}_{t,n}\uparrow\process(t)$ and $\overline{r}_{t,n}\downarrow\process(t)$ for any fixed $t\in\sits$.
Consider the computable nets of rational numbers defined by $\delta_{s,n}\coloneqq\overline{r}_{s,n}-\underline{r}_{s,n}\geq0$ and $r_{s,n}\coloneqq\group{\underline{r}_{s,n}+\overline{r}_{s,n}}/2$.
For any fixed $t\in\sits$, the sequence $\delta_{t,n}\downarrow0$, which implies that for any $N\in\naturalswithzero$ there is a natural number $e(t,N)$ such that $\delta_{t,n}\leq2^{-N}$ for all $n\geq e(t,N)$.
It is obvious that the function $e\colon\sits\times\naturalswithzero\to\naturalswithzero$ is computable, and that $n\geq e(s,N)$ also implies $\abs{\process(s)-r_{s,n}}\leq\abs{\overline{r}_{s,n}-\underline{r}_{s,n}}=\delta_{s,n}\leq2^{-N}$, for all $s\in\sits$ and $n,N\in\naturalswithzero$.
Hence, $\process$ is also computable.

We continue with the `only if' part.
Assume that $\process$ is computable, so there is a computable net of rational numbers $r_{s,n}$ and a computable function $e\colon\sits\times\naturalswithzero\to\naturalswithzero$ such that
$n\geq e(s,N)$ implies $\abs{r_{s,n}-\process(s)}\leq2^{-N}$ for all $s\in\sits$ and $n,N\in\naturalswithzero$. 
We prove that $\process$ is {\lscomp}; the proof that $\process$ is {\uscomp} is completely similar.
Consider the computable net of rational numbers $r'_{s,n}\coloneqq r_{s,e(s,n+1)}-2^{-\group{n+1}}$, then clearly $r'_{s,n}\leq\process(s)$ and $\abs{r'_{s,n}-\process(s)}<2^{-n}$ for all $s\in\sits$ and $n\in\naturalswithzero$.
This implies that we can always assume without loss of generality from the outset that $r_{s,n}\leq\process(s)$ and $e(s,N)=N$ [a similar idea was used in the proof of Proposition~\ref{prop:computable:simplified}].
We now construct a new computable net of rational numbers $\underline{r}_{s,n}$ from our original net: for any fixed $s\in\sits$, start with $\underline{r}_{s,0}\coloneqq r_{s,0}$ and $\underline{e}(s,0)=0$; let $\underline{e}(s,1)$ be the first $k>\underline{e}(s,0)$ such that $r_{s,k}\geq\underline{r}_{s,0}$, and let $\underline{r}_{s,1}\coloneqq r_{s,\underline{e}(s,1)}$; let $\underline{e}(s,2)$ be the first $k>\underline{e}(s,1)$ such that $r_{s,k}\geq\underline{r}_{s,1}$, and let $\underline{r}_{s,2}\coloneqq r_{s,\underline{e}(s,2)}$; and so on.
The function $\underline{e}\colon\sits\times\naturalswithzero\to\naturalswithzero$ is clearly computable, and therefore the net of rational numbers $\underline{r}_{s,n}$ is computable.
Moreover, for any fixed $t\in\sits$ the subsequence $\underline{r}_{t,n}$ of the sequence $r_{t,n}$ is non-decreasing by construction, and it converges to $\process(t)$ because the original sequence $r_{t,n}$ does.
\end{proof}\vspace{-10pt}

The following definitions and results are obvious and immediate. 
A gamble $f$ on $\outcomes$ is called \emph{computable} if and only if both its values $f(0)$ and $f(1)$ are computable real numbers.
An interval forecast $I=\pinterval\in\imprecisefrcsts$ is called \emph{computable} if and only if both its lower bound $\lp$ and upper bound $\up$ are computable real numbers.
A forecasting system $\frcstsystem$ is called \emph{computable} if the associated real processes $\lfrcstsystem$ and $\ufrcstsystem$ are.
Finally, a process difference $\adddelta\process$ is called (lower/upper semi)\emph{computable} if the real processes $\adddelta(s)(0)$ and $\adddelta(s)(1)$, $s\in\sits$ are; and similarly for a multiplier process $\multprocess$.

\begin{proposition}\label{prop:computable:expectation:processes}
For any computable gamble $f$ on $\outcomes$ and any computable forecasting system $\frcstsystem$, the real processes $\lex_{\frcstsystem(s)}(f)$ and $\uex_{\frcstsystem(s)}(f)$, $s\in\sits$ are computable.
\end{proposition}

\begin{proof}{\bf of Proposition~\ref{prop:computable:expectation:processes}\quad}
Observe that both the real process $\min\frcstsystem(s)f(1)+\sqgroup{1-\min\frcstsystem(s)}f(0)$, $s\in\sits$ and the real process $\max\frcstsystem(s)f(1)+\sqgroup{1-\max\frcstsystem(s)}f(0)$, $s\in\sits$ are computable. 
So are, therefore, their maximum process $\uex_{\frcstsystem(s)}(f)$, $s\in\sits$ and their minimum process $\lex_{\frcstsystem(s)}(f)$, $s\in\sits$.
\end{proof}\vspace{-20pt}

\begin{proposition}\label{prop:IcompIffGammaComp}
For any $I\in\imprecisefrcsts$, the stationary forecasting system $\frcstsystem_I$ is computable if and only if the interval $I$ is computable. 
\end{proposition}

\begin{proof}{\bf of Proposition~\ref{prop:IcompIffGammaComp}\quad}
Let $\lp\coloneqq\min I$ and $\up\coloneqq\max I$.
Obviously, the constant real processes $\lfrcstsystem_I(s)\coloneqq\lp$ and $\ufrcstsystem_I(s)\coloneqq\up$ are computable if and only if their constant values $\lp$ and $\up$ are.
\end{proof}\vspace{-20pt}

\begin{proposition}\label{prop:computable:expectations}
For any computable gamble $f$ on $\outcomes$ and any computable interval forecast $I=\pinterval\in\imprecisefrcsts$, the lower and upper expectations $\lex_I(f)$ and $\uex_I(f)$ are computable real numbers.
\end{proposition}

\begin{proof}{\bf of Proposition~\ref{prop:computable:expectations}\quad}
This is an immediate consequence of Propositions~\ref{prop:computable:expectation:processes} and~\ref{prop:IcompIffGammaComp}.
Or, alternatively, observe that both real numbers $\lp f(1)+(1-\lp)f(0)$ and $\up f(1)+(1-\up)f(0)$ are computable. 
So are, therefore, their maximum $\uex_I(f)$ and minimum $\lex_I(f)$.
\end{proof}\vspace{-10pt}

We will also need to use the following basic results.

\begin{proposition}\label{prop:computable:from:delta}
Consider any real process $\process$, and its process difference $\adddelta\process$.
Then the following statements hold:
\begin{enumerate}[label=\upshape(\roman*),leftmargin=*,noitemsep,topsep=0pt]
\item if $\process(\init)$ and $\adddelta\process$ are {\lscomp} then so is $\process$;
\item if $\process(\init)$ and $\adddelta\process$ are {\uscomp} then so is $\process$;
\item $\process$ is computable if and only if $\process(\init)$ and $\adddelta\process$ are.
\end{enumerate}
\end{proposition}

\begin{proof}{\bf of Proposition~\ref{prop:computable:from:delta}\quad}
We only prove the third statement.
The proof for the first and second statements are similar to the proof of the `if' part of the third, but simpler.

There are a number of ways to prove the third statement, but we will use Proposition~\ref{prop:computable:simplified}.

For the `if' part, we assume that $\process(\init)$ and $\adddelta\process$ are computable.
This implies that there are a computable sequence of rational numbers $r_{\init,n}$ and two computable nets of rational numbers $r_{s,n}^x$ such that $\abs{\process(\init)-r_{\init,n}}\leq2^{-n}$ and $\abs{\adddelta\process(s)(x)-r_{s,n}^x}\leq2^{-n}$ for all $s\in\sits$, $n\in\naturalswithzero$ and $x\in\outcomes$.
We now define the computable net of rational numbers $r_{s,n}$ as follows: for any $s=(\xvalto{m})\in\sits$, where $m\in\naturalswithzero$, and any $n\in\naturalswithzero$, let
\begin{equation*}
r_{s,n}\coloneqq r_{\init,n}+\sum_{k=1}^mr_{(\xvalto{k-1}),n}^{x_k}.
\end{equation*}
Then, since also
\begin{equation*}
\process(s)=\process(\init)+\sum_{k=1}^m\adddelta\process(\xvalto{k-1})(x_k),
\end{equation*}
we see that
\begin{equation*}
\abs{\process(s)-r_{s,n}}
\leq\abs{\process(\init)-r_{\init,n}}
+\sum_{k=1}^m\abs[\big]{\adddelta\process(\xvalto{k-1})(x_k)-r_{(\xvalto{k-1}),n}^{x_k}}
\leq(m+1)2^{-n},
\end{equation*}
so if we define the (clearly) computable function $e$ by
\begin{equation*}
e(s,N)\coloneqq N+m\geq N+\log_2(m+1),
\end{equation*}
then $n\geq e(s,N)$ implies that $\abs{\process(s)-r_{s,n}}\leq2^{-N}$ for all $s\in\sits$ and $n\in\naturalswithzero$.
Hence, $\process$ is computable.

For the `only if' part, assume that $\process$ is computable.
Then definitely in particular also its value $\process(\init)$ in the initial situation $\init$, so it only remains to prove that the process difference $\adddelta\process$ is computable.
Consider, to this effect, any $x\in\outcomes$.
It follows from the computability of $\process$ that there is a computable net of rational numbers $r'_{s,n}$ such that $\abs{\process(s)-r'_{s,n}}\leq2^{-n}$ and $\abs{\process(sx)-r'_{sx,n}}\leq2^{-n}$ and therefore also 
\begin{equation*}
r'_{sx,n}-r'_{s,n}-2^{-\group{n-1}}
\leq\process(sx)-\process(s)
\leq r'_{sx,n}-r'_{s,n}+2^{-\group{n-1}}
\text{ for all $s\in\sits$ and $n\in\naturalswithzero$.}
\end{equation*}
If we now let $r_{s,n}^x\coloneqq r'_{sx,n+1}-r'_{s,n+1}$, then this defines a computable net of rational numbers $r_{s,n}^x$ that satisfies $\abs{\adddelta\process(s)(x)-r_{s,n}^x}\leq2^{-n}$ for all $s\in\sits$ and all $n\in\naturalswithzero$.
This means that $\adddelta\process$ is computable.
\end{proof}\vspace{-20pt}

\begin{proposition}\label{prop:computable:from:delta:productorsum}
Consider any computable real processes $\processtoo$, $\processtoo'$ and $\processalso$.
Consider the real process $\process$ with computable $\process(\init)$, such that $\adddelta\process(s)=\indsing{1}\processtoo(s)+\indsing{0}\processtoo'(s)+\processalso(s)$ or $\adddelta\process(s)=\processtoo(s)\adddelta\processalso(s)$.
Then $\process$ is computable as well.
\end{proposition}

\begin{proof}{\bf of Proposition~\ref{prop:computable:from:delta:productorsum}\quad}
This is an immediate consequence of Proposition~\ref{prop:computable:from:delta}, since the conditions imply that $\adddelta\process$ is computable.
\end{proof}\vspace{-20pt}

\begin{proposition}\label{prop:computable:from:multiplier}
Consider any multiplier process $\multprocess$, and the associated real process $\mint$.
Then the following implications hold:
\begin{enumerate}[label=\upshape(\roman*),leftmargin=*,noitemsep,topsep=0pt]
\item if $\multprocess$ is {\lscomp}, then so are $\adddelta\mint$ and $\mint$;
\item if $\multprocess$ is {\uscomp}, then so are $\adddelta\mint$ and $\mint$;
\item if $\multprocess$ is {\comp}, then so are $\adddelta\mint$ and $\mint$.
\end{enumerate}
\end{proposition}

\begin{proof}{\bf of Proposition~\ref{prop:computable:from:multiplier}\quad}
We only give the proof for the first statement. 
The proof for the second statement is completely similar, and the third statement then follows readily from the first and the second.

Assume that $\multprocess$ is {\lscomp}.
This implies that there are two computable nets of rational numbers $r_{s,n}^x$ such that $r_{s,n}^x\uparrow\multprocess(s)(x)$, for $x\in\outcomes$.
Since $\multprocess(s)(x)\geq0$, we may assume without loss of generality that $r_{s,n}^x\geq0$ too.
We now construct the computable net of rational numbers $r_{s,n}$ as follows: for any $s=(\xvalto{m})\in\sits$, where $m\in\naturalswithzero$, and any $n\in\naturalswithzero$, let
\begin{equation*}
r_{s,n}\coloneqq\prod_{k=0}^{m-1}r_{(\xvalto{k}),n}^{x_{k+1}}\geq0.
\end{equation*}
Then, since also
\begin{equation*}
\mint(s)=\prod_{k=0}^{m-1}\multprocess(\xvalto{k})(x_{k+1}),
\end{equation*}
we see that $r_{s,n}\uparrow\mint(s)$ for all $s\in\sits$, so $\mint$ is indeed {\lscomp}.
Next, we construct two computable nets of rational numbers $t_{s,n}^x$ as follows: for any $s\in\sits$ and $x\in\outcomes$, let
\begin{equation*}
t_{s,n}^x\coloneqq r_{s,n}[r_{s,n}^x-1].
\end{equation*}
Then, since also, by Equation~\eqref{eq:supermartingale:multiplier:differences},
\begin{equation*}
\adddelta\mint(s)(x)=\mint(s)[\multprocess(s)(x)-1]
\end{equation*}
and $r_{s,n}\geq0$, we see that $t_{s,n}^x\uparrow\adddelta\mint(s)(x)$ for all $s\in\sits$ and $x\in\outcomes$, so $\adddelta\mint$ is indeed {\lscomp}.
\end{proof}\vspace{-20pt}

\begin{proposition}\label{prop:computablemultiplier:from:delta}
Consider a multiplier process $\multprocess$ and the associated real process $\mint$. 
If $\mint$ is positive and computable, then so is $\multprocess$.
\end{proposition}

\begin{proof}{\bf of Proposition~\ref{prop:computablemultiplier:from:delta}\quad}
Since $\mint$ is positive, it follows trivially that $\multprocess$ is positive as well. Consider now any $x\in\outcomes$. 
Since $\mint$ is computable, it follows from Proposition~\ref{prop:computable:from:delta} that $\adddelta\mint$ is computable, and therefore, we know that $\adddelta\mint(s)(x)$ is computable as well. 
Hence, since $\mint$ is computable and positive, we find that
\begin{equation*}
\multprocess(s)=\frac{\mint(s)+\adddelta\mint(s)(x)}{\mint(s)}
\end{equation*}
is computable.
\end{proof}\vspace{-20pt}

\subsection{Proofs and additional material for Section~\ref{sec:randomness}}
\label{app:randomness}
We denote by $\comptestsupermartins[\frcstsystem]$ the countable set of all computable test supermartingales for the forecasting system~$\frcstsystem$.

\begin{proof}{\bf of Proposition~\ref{prop:vacuous}\quad}
In the imprecise probability tree associated with the vacuous forecasting system~$\frcstsystem_\mathrm{v}$, a real process is a supermartingale if and only if it is non-increasing.
All test supermartingales are therefore bounded above by $1$ on any path $\pth\in\pths$. 
\end{proof}\vspace{-10pt}

\begin{proof}{\bf of Proposition~\ref{prop:nestedfrcstsystems}\quad}
Since $\frcstsystem\subseteq\frcstsystem^*$ implies that $\comptestsupermartins[\frcstsystem^*]\subseteq\comptestsupermartins[\frcstsystem]$, this follows trivially from Definition~\ref{def:randomness}.
\end{proof}\vspace{-20pt}

\begin{proposition}\label{prop:impreciseC:equivalence}
Consider any forecasting system $\frcstsystem$. 
Then for any outcome sequence $\pth$, the following statements are equivalent:
\begin{enumerate}[label=\upshape(\roman*),leftmargin=*,noitemsep,topsep=0pt]
\item\label{item:impreciseC:supermartin} $\sup_{n\in\naturals}\test(\pth^n)=+\infty$ for some computable non-negative supermartingale $\test$;
\item\label{item:impreciseC:testsupermartin} $\sup_{n\in\naturals}\test(\pth^n)=+\infty$ for some computable test supermartingale $\test$;
\item\label{item:impreciseC:testadddelta} $\sup_{n\in\naturals}\test(\pth^n)=+\infty$ for some test supermartingale $\test$, with $\adddelta\test$ computable;
\item\label{item:impreciseC:testmultdelta} $\sup_{n\in\naturals}\test(\pth^n)=+\infty$ for some test supermartingale $T=\mint$, with $\multprocess$ computable.
\end{enumerate}
\end{proposition}

\begin{proof}{\bf of Proposition~\ref{prop:impreciseC:equivalence}\quad}
Proposition~\ref{prop:computable:from:multiplier} implies that \ref{item:impreciseC:testmultdelta}$\then$\ref{item:impreciseC:testadddelta} and, since $T(\init)=1$ is rational and therefore computable, Proposition~\ref{prop:computable:from:delta} implies that \ref{item:impreciseC:testadddelta}$\then$\ref{item:impreciseC:testsupermartin}. 
Hence, since \ref{item:impreciseC:testsupermartin}$\then$\ref{item:impreciseC:supermartin} holds trivially, it suffices to prove that \ref{item:impreciseC:supermartin}$\then$\ref{item:impreciseC:testmultdelta}.

So consider any computable non-negative supermartingale $\test$ such that $\sup_{n\in\naturals}\test(\pth^n)=+\infty$. 
We will prove that there is a computable supermartingale multiplier $\multprocess$ such that $\sup_{n\in\naturals}\mint(\pth^n)=+\infty$. 
Let $\test'\coloneqq\nicefrac{1}{\alpha}(1+\test)$, with $\alpha\coloneqq1+\test(\init)$. 
Then $\alpha$ is clearly computable, and therefore, $\test'$ is computable as well. 
Also, since $\test$ is non-negative, we find that $\test'\geq\nicefrac{1}{\alpha}>0$ and $\test'(\init)=1$. 
Furthermore, since $\test$ is a supermartingale, $\test'$ is clearly a supermartingale as well. 
Finally, since $\smash{\sup_{n\in\naturals}\test(\pth^n)=+\infty}$, we have that $\smash{\sup_{n\in\naturals}\test'(\pth^n)=+\infty}$ as well. 
Hence, we find that $\test'$ is a computable test supermartingale such that $\test'>0$ and $\smash{\sup_{n\in\naturals}\test'(\pth^n)=+\infty}$. 
Hence, without loss of generality, we may assume that $\test$ is a positive test supermartingale. 
Since $\test$ is a positive test supermartingale, there is a unique supermartingale multiplier $\multprocess$ such that $\test=\mint$. 
Since $\test$ is positive and computable, the computability of $\multprocess$ follows from Proposition~\ref{prop:computablemultiplier:from:delta}.
\end{proof}\vspace{-16pt}

\begin{proposition}\label{prop:preciseC:equivalence}
Consider a computable precise forecasting system $\frcstsystem$. 
Then for any outcome sequence $\pth$, the following statements are equivalent:
\begin{enumerate}[label=\upshape(\roman*),leftmargin=*,noitemsep,topsep=0pt]
\item\label{item:preciseC:martin} $\sup_{n\in\naturals}\test(\pth^n)=+\infty$ for some computable non-negative martingale $\test$;
\item\label{item:preciseC:supermartin} $\sup_{n\in\naturals}\test(\pth^n)=+\infty$ for some computable non-negative supermartingale $\test$.
\end{enumerate}
\end{proposition}

\begin{proof}{\bf of Proposition~\ref{prop:preciseC:equivalence}\quad}
Since \ref{item:preciseC:martin}$\then$\ref{item:preciseC:supermartin} holds trivially, it suffices to prove that \ref{item:preciseC:supermartin}$\then$\ref{item:preciseC:martin}. 
So consider any computable non-negative supermartingale $\test'$ such that $\sup_{n\in\naturals}\test'(\pth^n)=+\infty$. 
We will prove that there is a computable non-negative martingale $\test$ such that $\sup_{n\in\naturals}\test(\pth^n)=+\infty$. 

Since $\test'$ is a computable non-negative supermartingale such that $\sup_{n\in\naturals}\test'(\pth^n)=+\infty$, it follows from Proposition~\ref{prop:impreciseC:equivalence} that there is a test supermartingale $\test''$, with $\adddelta\test''$ computable, such that $\sup_{n\in\naturals}\test''(\pth^n)=+\infty$. 
Now let $\test$ be the unique real process such that $\test(\init)=1$ and, for all $s\in\sits$:
\begin{align*}
\adddelta\test(s)(x)
\coloneqq&
\adddelta\test''(s)(x)-\ex_{\frcstsystem(s)}\big(\adddelta\test''(s)\big)\\
=&\adddelta\test''(s)(x)
-\frcstsystem(s)\adddelta\test''(s)(1)-\big(1-\frcstsystem(s)\big)\adddelta\test''(s)(0)
\text{ for all $x\in\outcomes$.}
\end{align*}
Since $\adddelta\test''$ and $\frcstsystem$ are computable, $\adddelta\test$ is clearly computable as well. 
Therefore, and because $\test(\init)$ is rational and therefore also computable, it follows from Proposition~\ref{prop:computable:from:delta} that $\test$ is computable. 
Furthermore, for any situation $s\in\sits$, we have that $\adddelta\test(s)\geq\adddelta\test''(s)$ because by assumption $\ex_{\frcstsystem(s)}(\adddelta\test''(s))\leq0$, and
\begin{align*}
\ex_{\frcstsystem(s)}\big(\adddelta\test(s)\big)
=\ex_{\frcstsystem(s)}\big(\adddelta\test''(s)-\ex_{\frcstsystem(s)}\big(\adddelta\test''(s)\big)\big)
=\ex_{\frcstsystem(s)}\big(\adddelta\test''(s)\big)-\ex_{\frcstsystem(s)}\big(\adddelta\test''(s)\big)
=0.
\end{align*}
Hence, it follows that $\test$ is a martingale and that $\adddelta\test\geq\adddelta\test''$.
Since $\test(\init)=\test''(\init)=1$, the latter implies that $\test\geq\test''$. 
Since $\test''$ is non-negative and $\sup_{n\in\naturals}\test''(\pth^n)=+\infty$, this in turn implies that $\test$ is non-negative and that $\sup_{n\in\naturals}\test(\pth^n)=+\infty$. Since we already know that $\test$ is a computable martingale, this establishes the desired result.
\end{proof}\vspace{-16pt}

\subsection{Proofs and additional material for Section~\ref{sec:consistency}}
\label{app:consistency}

\begin{proof}{\bf of Theorem~\ref{thm:consistency}\quad}
Consider the event $A\coloneqq\cset{\pth\in\pths}{\pth\text{ is computably random for $\frcstsystem$}}$.
We have to show that there is a test supermartingale that converges to $+\infty$ on $A^c$.

For every $\pth$ in $A^c$, there is some computable test supermartingale $\test_\pth$ that becomes unbounded on $\pth$.
Since the $\test_\pth$, $\pth\in A^c$ are countable in number, we can consider some countable convex combination $\test$ of all of them, with non-zero coefficients.
This is again a test supermartingale, that becomes unbounded on all $\pth$ in $A^c$.

We now construct yet another test supermartingale $\test'$ that converges to $+\infty$ on all $\pth$ where $\test$ becomes unbounded.
The argument has by now become standard \citep[Lemma~3.1]{shafer2001}. 
For any $n\in\naturals$, the real process $\test^{(n)}$ defined by
\begin{equation*}
\test^{(n)}(s)
\coloneqq
\begin{cases}
2^n&\text{ if $\test(t)\geq 2^n$ for some precursor $t\precedes s$ of $s$}\\
\test(s)&\text{ otherwise}
\end{cases}
\quad\text{for all $s\in\sits$},
\end{equation*}
is again a test supermartingale.
So is therefore the countable convex combination $\test'\coloneqq\sum_{n\in\naturals}2^{-n}\test^{(n)}$.
It is clear that $\test'$ converges to $+\infty$ on all paths $\pth$ where $\test$ becomes unbounded.
\end{proof}\vspace{-10pt}

\begin{proof}{\bf of Corollary~\ref{cor:consistency}\quad}
Consider the forecasting system $\frcstsystem$ defined by
\begin{equation*}
\frcstsystem(x_1,\dots,x_n)\coloneqq I_{n+1}
\text{ for all $(x_1,\dots,x_n)\in\outcomes^n$ and all $n\in\naturalswithzero$}.
\end{equation*}
Then it follows from Theorem~\ref{thm:consistency} that in the imprecise probability tree associated with $\frcstsystem$, (strictly) almost all $\pth$ are computably random for this $\frcstsystem$.
\end{proof}\vspace{-10pt}

In order to state our next set of results, we require some additional notions. Consider a real process $F\colon\sits\to\reals$ and a \emph{selection process} $\selection\colon\sits\to\outcomes$, and use them to define the real process $\average{F}\colon\sits\to\reals$ as follows:
\begin{multline*}
\average{F}(x_1,\dots,x_n)\\
\coloneqq
\begin{cases}
0
&\text{ if $\sum_{k=0}^{n-1}\selection(x_1,\dots,x_k)=0$}\\
\dfrac{\sum_{k=0}^{n-1}\selection(x_1,\dots,x_k)
\sqgroup{F(x_1,\dots,x_k,x_{k+1})-F(x_1,\dots,x_k)}}
{\sum_{k=0}^{n-1}\selection(x_1,\dots,x_k)} 
&\text{ if $\sum_{k=0}^{n-1}\selection(x_1,\dots,x_k)>0$},
\end{cases}
\end{multline*}
for all $n\in\naturalswithzero$ and $(x_1,\dots,x_n)\in\sits$.

In particular, fix any gamble $f$ on $\outcomes$, and let, for all $n\in\naturalswithzero$ and $(x_1,\dots,x_n)\in\sits$:
\begin{equation*}
\submartin^\frcstsystem_f(x_1,\dots,x_n)
\coloneqq\sum_{k=1}^{n}\sqgroup[\big]{f(x_k)-\lex_{\frcstsystem(x_1,\dots,x_{k-1})}(f)}
\end{equation*}
then on the one hand
\begin{align}
\adddelta\submartin^\frcstsystem_f(x_1,\dots,x_n)(x_{n+1})
&\coloneqq\submartin^\frcstsystem_f(x_1,\dots,x_{n+1})-\submartin^\frcstsystem_f(x_1,\dots,x_n)=f(x_{n+1})-\lex_{\frcstsystem(x_1,\dots,x_n)}(f),
\label{eq:lln:local:increments}
\end{align}
so $\adddelta\submartin^\frcstsystem_f(x_1,\dots,x_n)=f-\lex_{\frcstsystem(x_1,\dots,x_n)}(f)$, and therefore, on the other hand
\begin{equation}\label{eq:lln:local:increments:expectation}
\lex_{\frcstsystem(x_1,\dots,x_n)}(\adddelta\submartin^\frcstsystem_f(x_1,\dots,x_n))
=\lex_{\frcstsystem(x_1,\dots,x_n)}(f)-\lex_{\frcstsystem(x_1,\dots,x_n)}(f)
=0.
\end{equation}
We conclude that the real process $\submartin^\frcstsystem_f$ is a submartingale, whose differences $\adddelta\submartin^\frcstsystem_f(x_1,\dots,x_n)$ are uniformly bounded, for instance by $\varnorm{f}\coloneqq\max f-\min f$.
Observe by the way that in this particular case:
\begin{multline}\label{eq:average:over:selected:increments}
\average{\submartin^\frcstsystem_f}(x_1,\dots,x_n)\\
\coloneqq
\begin{cases}
0
&\text{ if $\sum_{k=0}^{n-1}\selection(x_1,\dots,x_k)=0$}\\
\dfrac{\sum_{k=0}^{n-1}\selection(x_1,\dots,x_k)
\sqgroup[\big]{f(x_{k+1})-\lex_{\frcstsystem(x_1,\dots,x_k)}(f)}}
{\sum_{k=0}^{n-1}\selection(x_1,\dots,x_k)} 
&\text{ if $\sum_{k=0}^{n-1}\selection(x_1,\dots,x_k)>0$}.
\end{cases}
\end{multline}

\begin{proposition}\label{prop:computable:lln:submartingale}
If the gamble $f$ on $\outcomes$ and the forecasting system $\frcstsystem$ are computable, then so is the submartingale $\submartin^\frcstsystem_f$.
\end{proposition}

\begin{proof}{\bf of Proposition~\ref{prop:computable:lln:submartingale}\quad}
We use Proposition~\ref{prop:computable:from:delta:productorsum}:  $\submartin^\frcstsystem_f(\init)=0$ is obviously a computable real number; let $\processtoo$ be the computable constant real process $f(1)$, $\processtoo'$ the computable constant real process $f(0)$, and $\processalso$  the real process defined by $\processalso(s)=-\lex_{\frcstsystem(s)}(f)$ for all $s\in\sits$, computable by Proposition~\ref{prop:computable:expectations}.
\end{proof}
\noindent
We can now apply our law of large numbers for submartingale differences \citep[Theorem~7]{cooman2015:markovergodic} to get to the following result, which generalises Philip Dawid's well-known consistency result for Bayesian Forecasters~\citep[Theorem~1]{dawid1982:well:calibrated:bayesian}, to deal with imprecise assessments: 

\begin{theorem}[The well-calibrated imprecise Bayesian]\label{thm:well:calibrated}
Let $\frcstsystem\colon\sits\to\imprecisefrcsts$ be any forecasting system, let $\selection\colon\sits\to\outcomes$ be any selection process, and let $f$ be any gamble on $\outcomes$. Then
\begin{equation*}
\sum_{k=0}^n\selection(X_1,\dots,X_k)\to+\infty\then\liminf_{n\to+\infty}\average{\submartin^\frcstsystem_f}(X_1,\dots,X_n)\geq0
\end{equation*}
(strictly) almost surely, in the imprecise probability tree associated with the forecasting system $\frcstsystem$.
\end{theorem}
\noindent
We repeat the proof here, borrowed from one of our earlier papers~\citep[Theorem~7]{cooman2015:markovergodic} and suitably adapted to include results on computability, because it will next help us prove a related result---Theorem~\ref{thm:well:calibrated:general}---that will turn out to be crucial in establishing the main claim of this paper.
One important step in this proof is, stripped to its bare essentials, based on a surprisingly elegant and effective idea that goes back to \citet[Lemma~3.3]{shafer2001}.

\begin{proof}{\bf of Theorem~\ref{thm:well:calibrated}\quad}
Consider the events $D\coloneqq\cset{\pth\in\pths}{\lim_{n\to+\infty}\sum_{k=0}^{n-1}\selection(\pth^k)=+\infty}$ and the event $A\coloneqq\cset{\pth\in\pths}{\liminf_{n\to+\infty}\average{\submartin^\frcstsystem_f}(\pth^n)<0}$.
We have to show that there is some test supermartingale $\test$ that converges to $+\infty$ on the set $D\cap A$.
Let\/ $B\coloneqq\max\{1,\varnorm{f}\}>0$, then we infer from Equation~\eqref{eq:lln:local:increments:expectation} that $B$ is a uniform real bound on $\Delta\submartin^\frcstsystem_f$, meaning that $\abs{\Delta\submartin^\frcstsystem_f(s)}\leq B$ for all situations $s\in\sits$. 

For any $r\in\naturals$, let $A_r\coloneqq\cset[\big]{\pth\in\pths}{\liminf_{n\to+\infty}\average{\submartin^\frcstsystem_f}(\pth^n)<-\frac{1}{2^r}}$, then $A=\bigcup_{r\in\naturals}A_r$.
So fix any $r\in\naturals$ and consider any $\pth\in D\cap A_r$, then we have in particular that
\begin{equation*}
\liminf_{n\to+\infty}\,\average{\submartin^\frcstsystem_f}(\pth^{n})<-\frac{1}{2^r},
\end{equation*}
and therefore
\begin{equation*}
(\forall m\in\naturals)
(\exists n_m\geq m)
\average{\submartin^\frcstsystem_f}(\pth^{n_m})<-\frac{1}{2^r}=-\epsilon,
\end{equation*}
with $0<\epsilon\coloneqq\frac{1}{2^r}<B$. 
Consider now the test supermartingale $\process_\submartin$ of Lemma~\ref{lem:martinwlln}, with in particular $\submartin\coloneqq\submartin^\frcstsystem_f$ and $\xi=\xi_r\coloneqq\frac{\epsilon}{2B^2}=\frac{1}{2^{r+1}B^2}=\frac{1}{2^{r+1}B}\frac{1}{B}<\frac{1}{B}$.
We denote it by $\process^{(r)}$.
It follows from Lemma~\ref{lem:martinwlln} that
\begin{equation}\label{eq:slln:intermediate}
\process^{(r)}(\pth^{n_m})
\geq\exp\group[\bigg]{\frac{\epsilon^2}{4B^2}\sum_{k=0}^{n_m-1}\selection(\pth^{k})}
=\exp\group[\bigg]{\frac{1}{2^{2r+2}B^2}\sum_{k=0}^{n_m-1}\selection(\pth^{k})}
\text{ for all $m\in\naturals$}.
\end{equation}
Consider any real $R>0$ and $m\in\naturals$. 
Since also $\pth\in D$, we know that $\lim_{n\to+\infty}\sum_{k=0}^{n-1}\selection(\pth^{k})=+\infty$, so there is some natural number $m'\geq m$ such that $\exp\group[\big]{\frac{1}{2^{2r+2}B^2}\sum_{k=0}^{m'-1}\selection(\pth^{k})}>R$.
Hence it follows from the statement in \eqref{eq:slln:intermediate} that there is some $n_{m'}\geq m'\geq m$---whence $\sum_{k=0}^{n_{m'}-1}\selection(\pth^{k})\geq\sum_{k=0}^{m'-1}\selection(\pth^{k})$---such that
\begin{equation*}
\process^{(r)}(\pth^{n_{m'}})
\geq\exp\group[\bigg]{\frac{1}{2^{2r+2}B^2\sum_{k=0}^{n_{m'}-1}\selection(\pth^{k})}}
\geq\exp\group[\bigg]{\frac{1}{2^{2r+2}B^2\sum_{k=0}^{m'-1}\selection(\pth^{k})}}
>R,
\end{equation*}
which implies that $\limsup_{n\to+\infty}\process^{(r)}(\pth^n)=+\infty$.
Observe that for this test supermartingale, 
\begin{equation*}
\process^{(r)}(\xvalto{n})\leq\group{\frac{3}{2}}^{n}
\text{ for all $n\in\naturalswithzero$ and $\xvalto{n}\in\outcomes^{n}$}.
\end{equation*}

Now define the process $\smash{\test^\frcstsystem_f\coloneqq\sum_{r\in\naturals}w^{(r)}\process^{(r)}}$ as any countable convex combination of the $\process^{(r)}$ constructed above, with positive weights $w^{(r)}>0$ that sum to one.
This is a real process, because each term in the series for $\test^\frcstsystem_f(\xvalto{n})$ is non-negative, and moreover 
\begin{multline}\label{eq:auxiliary:supermartingale:bounded}
\test^\frcstsystem_f(\xvalto{n})\leq\sum_{r\in\naturals}w^{(r)}\process^{(r)}
\leq\sum_{r\in\naturals}w^{(r)}\group[\Big]{\frac{3}{2}}^{n}
=\group[\Big]{\frac{3}{2}}^{n}\\
\text{ for all $n\in\naturalswithzero$ and $(\xvalto{n})\in\outcomes^{n}$}.
\end{multline}
This process is also positive, has $\test^\frcstsystem_f(\init)=1$, and, for any $\pth\in D\cap A$, it follows from the argumentation above that there is some $r\in\naturals$ such that $\pth\in D\cap A_r$ and therefore
\begin{equation}\label{eq:auxiliary:supermartingale:unbounded}
\limsup_{n\to+\infty}\test^\frcstsystem_f(\pth^n)
\geq w^{(r)}\limsup_{n\to+\infty}\process^{(r)}(\pth^n)
=+\infty,
\end{equation}
so $\limsup_{n\to+\infty}\test^\frcstsystem_f(\pth^n)=+\infty$.

We now prove that $\smash{\test^\frcstsystem_f}$ is a supermartingale, and therefore also a test supermartingale.
Consider any $n\in\naturalswithzero$ and any $(\xvalto{n})\in\outcomes^n$, then we have to prove that $\lex_{\frcstsystem(\xvalto{n})}(-\Delta\test^\frcstsystem_f(\xvalto{n}))\geq0$.
Since it follows from the argumentation in the proof of Lemma~\ref{lem:martinwlln} that
\begin{equation*}
-\Delta\process^{(r)}(\xvalto{n})
=\frac{1}{2^{r+1}B^2}\process^{(r)}(\xvalto{n})\Delta\submartin^\frcstsystem_f(\xvalto{n})
\text{ for all $r\in\naturals$}, 
\end{equation*}
we see that
\begin{equation}\label{eq:delta:test}
-\Delta\test^\frcstsystem_f(\xvalto{n})
=-\sum_{r\in\naturals}w^{(r)}\Delta\process^{(r)}
=\Delta\submartin^\frcstsystem_f(\xvalto{n})
\underbrace{\sum_{r\in\naturals}\frac{w^{(r)}}{2^{r+1}B^2}\process^{(r)}(\xvalto{n})}_{\eqqcolon C(\xvalto{n})},
\end{equation}
where $C(\xvalto{n})\geq0$ must be a real number, because, using a similar argument as before
\begin{equation*}
C(\xvalto{n})
=\sum_{r\in\naturals}\frac{w^{(r)}}{2^{r+1}B^2}\process^{(r)}(\xvalto{n})
\leq L\sum_{r\in\naturals}w^{(r)}\process^{(r)}(\xvalto{n})
\leq L\group[\Big]{\frac{3}{2}}^{n}
\end{equation*}
for some real $L>0$.
Therefore indeed, using the non-negative homogeneity of lower expectations:
\begin{align*}
\lex_{\frcstsystem(\xvalto{n})}(-\Delta\test^\frcstsystem_f(\xvalto{n}))
&=\lex_{\frcstsystem(\xvalto{n})}(C(\xvalto{n})\Delta\submartin^\frcstsystem_f(\xvalto{n}))\\
&=C(\xvalto{n})\lex_{\frcstsystem(\xvalto{n})}(\Delta\submartin^\frcstsystem_f(\xvalto{n}))
\geq0,
\end{align*}
because $\submartin^\frcstsystem_f$ is a submartingale; see Equation~\eqref{eq:lln:local:increments:expectation}.

Since we now know that $\test^\frcstsystem_f$ is a supermartingale that is furthermore bounded below (by $0$) it follows from our supermartingale convergence theorem \citep{cooman2015:markovergodic} that there is some test supermartingale that converges to $+\infty$ on all paths where $\test^\frcstsystem_f$ does not converge to a real number, and therefore in particular on all paths in $D\cap A$.
Hence $D\cap A$ is indeed strictly null.
\end{proof}\vspace{-20pt}

\begin{lemma}\label{lem:martinwlln}
Consider any real $B>0$ and $0<\xi<\frac{1}{B}$.
Let\/ $\submartin$ be any submartingale such that $\vert\adddelta\submartin\vert\leq B$.
Let $\selection$ be any real process that only assumes values in $\outcomes$.
Then the process $\process_\submartin$ defined by:
\begin{multline}\label{eq:auxiliary:supermartingale}
\process_\submartin(\xvalto{n})
\coloneqq\prod_{k=0}^{n-1} 
\sqgroup[\big]{1-\xi\selection(\xvalto{k})\adddelta\submartin(\xvalto{k})(\xval{k+1})}\\
\text{for all $n\in\naturalswithzero$ and $(\xvalto{n})\in\outcomes^{n}$}
\end{multline}
is a positive supermartingale with $\process_\submartin(\init)=1$, and therefore in particular a test supermartingale.
Moreover, for $\xi\coloneqq\frac{\epsilon}{2B^2}$, with $0<\epsilon<B$, we have that 
\begin{multline*}
\average{\submartin^\frcstsystem_f}(\xvalto{n})\leq-\epsilon
\then
\process_\submartin(\xvalto{n})
\geq\exp\group[\bigg]{\frac{\epsilon^2}{4B^2}\sum_{k=0}^{n-1}\selection(\xvalto{k})}\\
\text{for all $n\in\naturalswithzero$ and $(\xvalto{n})\in\outcomes^{n}$}.
\end{multline*}
Finally, if $\xi$ and $\submartin$ are computable, and $\selection$ is computable, then $\process_\submartin$ is computable as well.
\end{lemma}

\begin{proof}{\bf of Lemma~\ref{lem:martinwlln}\quad}
$\process_\submartin(\init)=1$ trivially.
To prove that $\process_\submartin$ is positive, consider any $n\in\naturals$ and any $(\xvalto{n})\in\outcomes^{n}$.
Since it follows from $0<\xi B<1$, $\vert\adddelta\submartin\vert\leq B$ and $\selection\in\outcomes$ that $1-\xi\selection(\xvalto{k})\adddelta\submartin(\xvalto{k})(\xval{k+1})\geq 1-\xi B>0$ for all $0\leq k\leq n-1$, we see that indeed:
\begin{equation*}
\process_\submartin(\xvalto{n})
=\prod_{k=0}^{n-1}\sqgroup[\big]{1-\xi\selection(\xvalto{k})\adddelta\submartin(\xvalto{k})(\xval{k+1})}
>0.
\end{equation*}
This also tells us that if we let
\begin{equation}\label{eq:multprocess:submartingale}
\multprocess_\submartin(s)\coloneqq1-\xi\selection(s)\adddelta\submartin(s)>0,
\end{equation}
then $\multprocess_\submartin$ is a multiplier process, and $\process_\submartin=\mint[\multprocess_\submartin]$.
Moreover, since $\xi>0$ and $\selection(s)\in\outcomes$, we infer from the coherence and conjugacy properties of lower and upper expectations that 
\begin{align*}
\uex_{\frcstsystem(s)}(\multprocess_\submartin)
&=\uex_{\frcstsystem(s)}\group[\big]{1-\xi(s)\selection(s)\adddelta\submartin(s)}\\
&=1+\uex_{\frcstsystem(s)}\group[\big]{-\xi(s)\selection(s)\adddelta\submartin(s)}
=1-\xi(s)\selection(s)\lex_{\frcstsystem(s)}\group[\big]{\adddelta\submartin(s)}\leq1,
\end{align*}
where the inequality follows from $\lex_{\frcstsystem(s)}\group{\adddelta\submartin(s)}\geq0$, because we assumed that $\submartin$ is a submartingale.
This shows that $\multprocess_\submartin$ is a supermartingale multiplier, and therefore $\process_\submartin=\mint[\multprocess_\submartin]$ is indeed a test supermartingale.

For the second statement, consider any $0<\epsilon<B$ and let $\xi\coloneqq\frac{\epsilon}{2B^2}$. 
Then for any $n\in\naturalswithzero$ and $\xvalto{n}\in\outcomes^{n}$ such that $\average{\submartin^\frcstsystem_f}(\xvalto{n})\leq-\epsilon$, we have for all real $K$:
\begin{align}
\process_\submartin(\xvalto{n})\geq\exp(K) 
&\ifandonlyif
\prod_{k=0}^{n-1} 
\sqgroup[\big]{1-\xi\selection(\xvalto{k})\adddelta\submartin(\xvalto{k})(\xval{k+1})}\geq\exp(K)\notag\\
&\ifandonlyif
\sum_{k=0}^{n-1} 
\ln\sqgroup[\big]{1-\xi\selection(\xvalto{k})\adddelta\submartin(\xvalto{k})(\xval{k+1})}\geq K.\label{eq:K}
\end{align}
Since $\vert\adddelta\submartin\vert\leq B$, $\selection\in\outcomes$ and $0<\epsilon<B$, we find that
\begin{equation*}
-\xi\selection(\xvalto{k})\adddelta\submartin(\xvalto{k})
\geq-\xi B
=-\frac{\epsilon}{2B}
>-\frac{1}{2}
\quad\text{ for $0\leq k\leq n-1$}. 
\end{equation*} 
As $\ln(1+x)\ge x-x^2$ for $x>-\frac{1}{2}$, this allows us to infer that
\begin{multline*}
\sum_{k=0}^{n-1} 
\ln\sqgroup[\big]{1-\xi\selection(\xvalto{k})\adddelta\submartin(\xvalto{k})(\xval{k+1})}\\
\begin{aligned}
&\geq
\sum_{k=0}^{n-1} 
\sqgroup[\big]{-\xi\selection(\xvalto{k})\adddelta\submartin(\xvalto{k})(\xval{k+1})
-\xi^2\selection(\xvalto{k})^2(\adddelta\submartin(\xvalto{k})(\xval{k+1}))^2}\\
&=
-\xi\sum_{k=0}^{n-1}\selection(\xvalto{k})\average{\submartin^\frcstsystem_f}(\xvalto{n})
-\xi^2\sum_{k=0}^{n-1} 
\selection(\xvalto{k})(\adddelta\submartin(\xvalto{k})(\xval{k+1}))^2\\
&\geq\xi\sum_{k=0}^{n-1}\selection(\xvalto{k})\epsilon
-\xi^2\sum_{k=0}^{n-1}\selection(\xvalto{k})B^2\\
&=\xi(\epsilon-\xi B^2)\sum_{k=0}^{n-1}\selection(\xvalto{k})
=\frac{\epsilon^2}{4B^2}\sum_{k=0}^{n-1}\selection(\xvalto{k}),
\end{aligned}
\end{multline*}
where the first equality holds because $\selection^2=\selection$. 
Now choose $K\coloneqq\frac{\epsilon^2}{4B^2}\sum_{k=0}^{n-1}\selection(\xvalto{k})$ in Equation~\eqref{eq:K}.

We now prove the last statement, dealing with the computability of $\process_\submartin$.

Since $\submartin$ is assumed to be computable, so is $\adddelta\submartin$, by Proposition~\ref{prop:computable:from:delta}.
Since also $\xi$ and $\selection$ are assumed to be computable, we infer from Equation~\eqref{eq:multprocess:submartingale} that the multiplier process $\multprocess_\submartin$ is computable too.
If we now invoke Proposition~\ref{prop:computable:from:multiplier}, we find that $\process_\submartin=\mint[\multprocess_\submartin]$ is therefore computable as well.  
\end{proof}\vspace{-10pt}

\begin{proof}{\bf of Theorem~\ref{thm:well:calibrated:general}\quad}
Assume {\itshape ex absurdo} that the inequality is not satisfied.
Then we infer from the proof of Theorem~\ref{thm:well:calibrated} that there is some positive test supermartingale $\test^\frcstsystem_f=\sum_{r\in\naturals}w^{(r)}\process^{(r)}$ that is unbounded above on $\pth$: see Equation~\eqref{eq:auxiliary:supermartingale:unbounded} there.
We will now go back to the details of that proof to show that we can make sure that $\test^\frcstsystem_f=\mint$ for some computable multiplier process $\multprocess$, thereby contradicting the assumed computable randomness for $\frcstsystem$.

First of all, recall that for any $r\in\naturals$, the test supermartingale $\process^{(r)}$ is the test supermartingale $\process_\submartin$ constructed in Lemma~\ref{lem:martinwlln}, for the particular choices $\submartin=\submartin^\frcstsystem_f$, $B=\max\set{1,\varnorm{f}}$ and $\xi=\xi_r=\frac{1}{2^{r+1}B^2}$. 
Since the gamble $f$ is assumed to be computable, so is the real number $\varnorm{f}=\abs{f(1)-f(0)}$, and therefore also the real numbers $B$ and $\xi_r$.
We infer from Proposition~\ref{prop:computable:lln:submartingale} that the submartingale $\submartin^\frcstsystem_f$ is computable, because the forecasting system $\frcstsystem$ and the gamble $f$ are.
Since in addition $\selection$ is assumed to be computable, we infer from Lemma~\ref{lem:martinwlln} that the test supermartingale $\process^{(r)}$ is computable.

We now consider the version of the test supermartingale $\test^\frcstsystem_f=\sum_{r\in\naturals}w^{(r)}\process^{(r)}$ corresponding to the particular choices $w^{(r)}\coloneqq2^{-r}$, and prove that $\test^\frcstsystem_f$ is computable. 
To this effect, we use Proposition~\ref{prop:computable:closure}.
Indeed, consider, for each $s=(\xvalto{m})\in\sits$ and $n\in\naturalswithzero$, the real number
\vspace{-5pt}
\begin{equation*}
x_{s,n}\coloneqq\sum_{r=1}^n2^{-r}\process^{(r)}(s),
\end{equation*}
which is computable because all real processes $\process^{(r)}$ are.
Since $0\leq\process^{(r)}(s)\leq\group[\big]{\frac{3}{2}}^m$, we get
\begin{equation*}
\abs{\test^\frcstsystem_f(s)-x_{s,n}}
=\sum_{r=n+1}^{+\infty}2^{-r}\process^{(r)}(s)
\leq\group[\Big]{\frac{3}{2}}^m\sum_{r=n+1}^{+\infty}2^{-r}
=\group[\Big]{\frac{3}{2}}^m2^{-n}.
\end{equation*}
If we now let $e(s,N)\coloneqq N+m$, then since
\begin{equation*}
N+m\geq N+m\log_2\frac{3}{2},
\end{equation*}  
we see that $n\geq e(s,N)$ implies $\abs{\test^\frcstsystem_f(s)-x_{s,n}}\leq2^{-N}$ for all $s\in\sits$ and $n,N\in\naturalswithzero$.
Since the function $e$ is clearly computable, we may use Proposition~\ref{prop:computable:closure} to conclude that $\test^\frcstsystem_f$ is indeed computable.

Now let $\multprocess$ be the unique supermartingale multiplier such that $\test^\frcstsystem_f=\mint$ [the uniqueness follows from the fact that $\test^\frcstsystem_f$ is positive]. Since $\test^\frcstsystem_f$ is computable, it then follows from Proposition~\ref{prop:computablemultiplier:from:delta} that $\multprocess$ is computable as well.
\end{proof}\vspace{-20pt}

\subsection{Proofs and additional material for Section~\ref{sec:constantintervalforecasts}}
\label{app:constantintervalforecasts}

\begin{proof}{\bf of Proposition~\ref{prop:constantcalibrated:top}\quad}
Immediate consequence of Proposition~\ref{prop:vacuous}, with $\constantfrcst{\frcsts}=\frcstsystem_\mathrm{v}\in\random{\mathrm{C}}$.
\end{proof}

\begin{proof}{\bf of Proposition~\ref{prop:constantcalibrated:increasing}\quad}
This follows from Proposition~\ref{prop:nestedfrcstsystems}, because $I\subseteq J$ implies $\constantfrcst{I}\subseteq\constantfrcst{J}$.
\end{proof}

\begin{proof}{\bf of Corollary~\ref{cor:well:calibrated:constant}\quad}
First, assume that $I$ is computable.
It follows from Proposition~\ref{prop:IcompIffGammaComp} that $\constantfrcst{I}$ is computable as well. 
Furthermore, if we let $\indsing{1}(x)\coloneqq x$ for all $x\in\outcomes$, then $\indsing{1}$ and $-\indsing{1}$ are clearly computable gambles on $\outcomes$. 
The first and last inequality now follow from Theorem~\ref{thm:well:calibrated:general}, by choosing $f=\indsing{1}$ and $f=-\indsing{1}$, respectively, since $\lex_I(\indsing{1})=\lp$ and $\lex_I(-\indsing{1})=-\up$. 
The second inequality is a standard property of limits inferior and superior.

If $I$ is not computable, then for any $\epsilon>0$, since all rational numbers are computable, there is some computable $J=[\underline{q},\overline{q}]\in\imprecisefrcsts$ such that $\smash{\lp-\epsilon\leq\underline{q}\leq\lp\leq\up\leq\overline{q}\leq\up+\epsilon}$. 
Since $I\subseteq J$, it follows from Proposition~\ref{prop:constantcalibrated:increasing} that also $J\in\constantrandom{\mathrm{C}}$. 
Since, moreover, $J$ is computable, it follows from the first part of the proof that
\begin{equation*}
\lp-\epsilon\leq\underline{q}
\leq\liminf_{n\to+\infty}
\frac{\sum_{k=0}^{n-1}S(x_1,\dots,x_k)x_{k+1}}{\sum_{k=0}^{n-1}S(x_1,\dots,x_k)}
\leq\limsup_{n\to+\infty}
\frac{\sum_{k=0}^{n-1}S(x_1,\dots,x_k)x_{k+1}}{\sum_{k=0}^{n-1}S(x_1,\dots,x_k)}
\leq\overline{q}\leq\up+\epsilon.
\end{equation*}
Since $\epsilon>0$ is arbitrary, this completes the proof.
\end{proof}\vspace{-10pt}

\begin{proof}{\bf of Proposition~\ref{prop:constantcalibrated:nonempty:intersection:inside:ML}\quad}
For the first statement, let $I=\pinterval$ and $J=\qinterval$ and assume {\itshape ex absurdo} that $I\cap J=\emptyset$.
We may assume without loss of generality that $\lp>\uptoo$. 
It then follows from Corollary~\ref{cor:well:calibrated:constant}, with $S(s)\coloneqq1$ for all $s\in\sits$, that
\begin{equation*}
\liminf_{n\to+\infty}\frac{1}{n}\sum_{k=1}^nx_k
\leq
\limsup_{n\to+\infty}\frac{1}{n}\sum_{k=1}^nx_k
\leq
\overline{q}
<
\underline{p}
\leq
\liminf_{n\to+\infty}\frac{1}{n}\sum_{k=1}^nx_k,
\end{equation*}
a contradiction.

For the second statement, let $K\coloneqq I\cap J$.
We will prove that $K\in\constantrandom{\mathrm{C}}$. 
Again, let $I=\pinterval$ and $J=\qinterval$. 
Because of symmetry, we may assume without loss of generality that $\smash{\lptoo\leq\lp}$. Furthermore, due to the first statement, we know that then $\lp\leq\uptoo$. 
If we have that $I\subseteq J$, then $I=I\cap J$ and therefore, since $I\in\constantrandom{\mathrm{C}}$, the result holds trivially. Hence, we may assume without loss of generality that $\smash{\lptoo<\lp\leq\uptoo<\up}$, which implies that $K=I\cap J=[\lp,\uptoo]$.

Consider any computable test supermartingale $\test$ in $\comptestsupermartins[\constantfrcst{K}]$, then we must show that $\test$ remains bounded on $\pth$. 
We may assume without loss of generality that there is some computable supermartingale multiplier $\multprocess$ for $\constantfrcst{K}$ such that $\test=\mint$. 

Now let $\multprocess_I$ be the map from situations to gambles on $\outcomes$, defined by
\begin{equation*}
\multprocess_I(s)(z)
\coloneqq
\begin{cases}
\min\set{\multprocess(s)(1),1}
&\text{if $z=1$}\\
\max\set{\multprocess(s)(0),1}
&\text{if $z=0$}\\
\end{cases}
\quad\text{for all $s\in\sits$ and $z\in\outcomes$.}
\end{equation*}
Then $\multprocess_I$ is a supermartingale multiplier for $\constantfrcst{I}$: that it is  non-negative follows from the non-negativity of $\multprocess$; moreover, for any $s\in\sits$, we have that $\uex_I(\multprocess_I(s))\leq1$. 
Indeed, to prove the this, we consider two cases: $\multprocess(s)(0)\leq1$ and $\multprocess(s)(0)>1$. 
If $\multprocess(s)(0)\leq1$, then $\multprocess_I(s)\leq1$ and therefore also $\uex_{I}(\multprocess_I(s))\leq1$, by~\ref{axiom:coherence:bounds}. 
The case that $\multprocess(s)(0)>1$ is a bit more involved. 
For a start, since $\multprocess(s)(0)>1$ implies that $\multprocess_I(s)(1)<\multprocess_I(s)(0)$, we find that
\begin{equation*}
\uex_{I}(\multprocess_I(s))
=\ex_{\lp}(\multprocess_I(s))
=\uex_{K}(\multprocess_I(s)).
\end{equation*}
Furthermore, since by assumption $\uex_K(\multprocess(s))\leq1$, $\multprocess(s)(0)>1$ implies that $\multprocess(s)(1)\leq1$, again by~\ref{axiom:coherence:bounds}.
We therefore find that $\multprocess(s)=\multprocess_I(s)$. 
By combining these two findings, it follows that indeed here also
\begin{equation*}
\uex_{I}(\multprocess_I(s))
=\uex_{K}(\multprocess_I(s))
=\uex_{K}(\multprocess(s))
\leq1.
\end{equation*}
Since $\multprocess_I$ is indeed a supermartingale multiplier for $\constantfrcst{I}$, we find that $\test_I\coloneqq\mint[\multprocess_I]$ is a test supermartingale for $\constantfrcst{I}$.

Furthermore, since $\multprocess$ is computable, so is $\multprocess_I$, because taking minima and maxima preserves computability. 
Hence, $\test_I$ belongs to $\comptestsupermartins[\constantfrcst{I}]$. 
Therefore, and because $I\in\constantrandom{\mathrm{C}}$, it follows that $\test_I(\pth^n)$ remains bounded as $n\to+\infty$.

Also, if we let $\multprocess_J$ be a map from situations to gambles on $\outcomes$, defined by
\begin{equation*}
\multprocess_J(s)(z)
\coloneqq
\begin{cases}
\max\{\multprocess(s)(1),1\}
&\text{if $z=1$}\\
\min\{\multprocess(s)(0),1\}
&\text{if $z=0$}
\end{cases}
\quad\text{for all $s\in\sits$ and $z\in\outcomes$},
\end{equation*}
and consider $\test_J\coloneqq\mint[\multprocess_J]$, a similar course of reasoning leads us to conclude that $\test_J$ belongs to $\comptestsupermartins[\constantfrcst{J}]$. 
Therefore, and because $J\in\constantrandom{\mathrm{C}}$, it follows that $\test_J(\pth^n)$ remains bounded as $n\to+\infty$.

Next, we observe that $\multprocess=\multprocess_I\multprocess_J$, and therefore also $\test=\mint=\mint[\multprocess_I]\mint[\multprocess_J]=\test_I\test_J$.
And since both $\test_I$ and $\test_J$ remain bounded on $\pth$, so, therefore, does $\test$.
\end{proof}\vspace{-16pt}

\begin{proof}{\bf of Proposition~\ref{prop:constantcalibrated:computable:sequence}\quad}
When $\pth$ is computable and has infinitely many zeroes and ones, there is a computable selection process $\selection_1$ that selects all the ones, with $\sum_{k=0}^{n-1}\selection_1(x_1,\dots,x_k)\to+\infty$, and another computable selection process $\selection_0=1-\selection_1$ that selects all the zeroes, with $\sum_{k=0}^{n-1}\selection_0(x_1,\dots,x_k)\to+\infty$.
For any $I\in\constantrandom{\mathrm{C}}$, we then infer from Corollary~\ref{cor:well:calibrated:constant} that
\begin{equation*}
\min I
\leq\liminf_{n\to+\infty}\frac{\sum_{k=0}^{n-1}\selection_0(x_1,\dots,x_k)x_{k+1}}
{\sum_{k=0}^{n-1}\selection_0(x_1,\dots,x_k)}
=0,
\end{equation*}
and similarly
\begin{equation*}
\max I
\geq\limsup_{n\to+\infty}\frac{\sum_{k=0}^{n-1}\selection_1(x_1,\dots,x_k)x_{k+1}}
{\sum_{k=0}^{n-1}\selection_1(x_1,\dots,x_k)}
=1,
\end{equation*}
so indeed $I=[0,1]$.
\end{proof}\vspace{-16pt}

\subsection{Proofs and additional material for Section~\ref{sec:inherently}}
\label{app:inherently}

\begin{proof}{\bf of Proposition~\ref{prop:first:example}\quad}
The converse implication follows at once from Proposition~\ref{prop:nestedfrcstsystems} and the fact that for any $I\in\imprecisefrcsts$ such that $[p,q]\subseteq I$, the stationary forecasting system $\constantfrcst{I}$ is more conservative than $\frcstsystem_{p,q}$, in the sense that $\frcstsystem_{p,q}\subseteq\constantfrcst{I}$.

For the direct implication, assume that $I\in\constantrandom{\mathrm{C}}$ and fix any $\epsilon>0$. Since all rational numbers are computable, there are computable intervals $\smash{[\underline{p},\overline{p}]\in\imprecisefrcsts}$ and $\smash{[\underline{q},\overline{q}]\in\imprecisefrcsts}$ such that
\begin{equation*}
p\in[\underline{p},\overline{p}]\subseteq[p-\epsilon,p+\epsilon]
\text{ and }
q\in[\underline{q},\overline{q}]\subseteq[q-\epsilon,q+\epsilon].
\end{equation*} 
Consider now the forecasting system $\frcstsystem_\epsilon$, defined by
\begin{equation*}
\frcstsystem_\epsilon(z_1,\dots,z_n)
\coloneqq
\begin{cases}
[\underline{p},\overline{p}] &\text{if $n$ is odd}\\
[\underline{q},\overline{q}] &\text{if $n$ is even}
\end{cases}
\quad\text{for all $(z_1,\dots,z_n)\in\sits$.}
\end{equation*}
Then $\frcstsystem_\epsilon$ is clearly computable and, since $\frcstsystem_{p,q}\subseteq\frcstsystem_\epsilon$, we know from Proposition~\ref{prop:nestedfrcstsystems} that $\pth$ is computably random for $\frcstsystem_\epsilon$. 
Therefore, we find that
\begin{equation*}
\min I
\leq\liminf_{n\to+\infty}\frac{\sum_{k=1}^nx_{2k}}{n}
\leq\limsup_{n\to+\infty}\frac{\sum_{k=1}^nx_{2k}}{n}
\leq\overline{p}
\leq p+\epsilon
\end{equation*}
where the first and third inequality follow from Corollary~\ref{cor:well:calibrated:constant} and Theorem~\ref{thm:well:calibrated:general}, respectively, for appropriately chosen computable selection processes.
Similarly, we also find that
\begin{equation*}
\max I
\geq\limsup_{n\to+\infty}\frac{\sum_{k=1}^nx_{2k-1}}{n}
\geq\liminf_{n\to+\infty}\frac{\sum_{k=1}^nx_{2k-1}}{n}
\geq\underline{q}
\geq q-\epsilon.
\end{equation*}
Since $\epsilon>0$ is arbitrary, this allows us to conclude that $\min I\leq p$ and $\max I\geq q$, and, therefore, that $[p,q]\subseteq I$.
\end{proof}\vspace{-10pt}

That $\delta_n\downarrow0$ follows readily by combining Equations~\eqref{eq:pnoneminus:one} and~\eqref{eq:pnoneminus:two} below.
Also observe, for all $n\in\naturals$, that
\begin{align}
p_n(1-p_n)
&=\group[\Big]{\frac{1}{2}-\delta_n}\group[\Big]{\frac{1}{2}+\delta_n}
=\frac{1}{4}-\delta_n^2
\label{eq:pnoneminus:one}\\
&=\frac{1}{4}-e^{-\frac{2}{n+1}}\group[\big]{e^{\frac{1}{n+1}}-1}
=\frac{1}{4}-e^{-\frac{1}{n+1}}+e^{-\frac{2}{n+1}}
=\group[\Big]{e^{-\frac{1}{n+1}}-\frac{1}{2}}^2,
\label{eq:pnoneminus:two}
\end{align}
and therefore, since $e^{-\frac{1}{n+1}}>\frac{1}{2}$, that
\begin{equation}\label{eq:divergence:example:one}
e^{\frac{1}{2(n+1)}}\frac{1}{\sqrt2}\left(\sqrt{p_n}+\sqrt{1-p_n}\right)
=e^{\frac{1}{2(n+1)}}\sqrt{\frac{1}{2}+\sqrt{p_n(1-p_n)}}
=e^{\frac{1}{2(n+1)}}\sqrt{e^{-\frac{1}{n+1}}}=1.
\end{equation}

\begin{proof}{\bf of Proposition~\ref{prop:second:example}\quad}
We first show that $\set{\nicefrac{1}{2}}\notin\constantrandom{\mathrm{C}}$, so the sequence is not computably random. 
Consider the multiplier processes $\multprocess_{\nicefrac{1}{2}}$ and $\multprocess_{\sim\nicefrac{1}{2}}$, defined for all $n\in\naturals$ and $(z_1,\dots,z_{n-1})\in\sits$ by
\begin{equation*}
\multprocess_{\nicefrac{1}{2}}(z_1,\dots,z_{n-1})
\coloneqq e^{\frac{1}{2(n+1)}}\sqrt{2\pmass[n]}
\text{ and }
\multprocess_{\sim\nicefrac{1}{2}}(z_1,\dots,z_{n-1})
\coloneqq e^{\frac{1}{2(n+1)}}\frac{1}{\sqrt{2\pmass[n]}},
\end{equation*}
where we define, for any $p\in\frcsts$, the corresponding mass function $\pmass$ by letting $\pmass(1)\coloneqq p$ and $\pmass(0)\coloneqq1-p$. We then find that
\begin{align}\label{eq:explodes:too}
&\ln\big(\mint[\multprocess_{\nicefrac{1}{2}}](\xvalto{n})
\mint[\multprocess_{\sim\nicefrac{1}{2}}](\xvalto{n})\big)
=\sum_{k=1}^n\frac{1}{k+1}\to+\infty
\text{ as $n\to+\infty$}.
\end{align}
Furthermore, taking into account Equation~\eqref{eq:divergence:example:one}, it is then easy to verify that for all $n\in\naturals$ and $\smash{(z_1,\dots,z_{n-1})\in\sits}$:
\begin{equation*}
\ex_{\nicefrac{1}{2}}\big(\multprocess_{\nicefrac{1}{2}}(z_1,\dots,z_{n-1})\big)
=\frac{1}{2}e^{\frac{1}{2(n+1)}}\sqrt{2p_{n}}
+\frac{1}{2}e^{\frac{1}{2(n+1)}}\sqrt{2(1-p_{n})}
=1
\end{equation*}
and
\begin{equation*}
\ex_{p_n}\big(\multprocess_{\sim\nicefrac{1}{2}}(z_1,\dots,z_{n-1})\big)
=p_ne^{\frac{1}{2(n+1)}}\frac{1}{\sqrt{2p_n}}
+(1-p_n)e^{\frac{1}{2(n+1)}}\frac{1}{\sqrt{2(1-p_n)}}
=1.
\end{equation*}
Hence, we find that $\multprocess_{\nicefrac{1}{2}}$ is a supermartingale multiplier for the stationary forecasting system $\constantfrcst{\set{\nicefrac{1}{2}}}$ and that $\multprocess_{\sim\nicefrac{1}{2}}$ is a supermartingale multiplier for the forecasting system $\xmplfrcstsystem$. 
Both of these supermartingale multipliers are furthermore clearly computable.
Since $\multprocess_{\sim\nicefrac{1}{2}}$ is a computable supermartingale multiplier for the forecasting system $\xmplfrcstsystem$, it follows by assumption that $\mint[\multprocess_{\sim\nicefrac{1}{2}}](\xvalto{n})$ remains bounded as $n\to+\infty$. 
Therefore, and taking into account Equation~\eqref{eq:explodes:too}, we find that $\mint[\multprocess_{\nicefrac{1}{2}}](\xvalto{n})\to+\infty$ as $n\to+\infty$. 
Since $\multprocess_{\nicefrac{1}{2}}$ is a computable supermartingale multiplier for the stationary forecasting system $\constantfrcst{\set{\nicefrac{1}{2}}}$, this implies that indeed $\set{\nicefrac{1}{2}}\notin\constantrandom{\mathrm{C}}$.

In a similar way, for all $\epsilon>0$, it can also be shown that $\sqgroup{\nicefrac{1}{2}-\epsilon,\nicefrac{1}{2}}\notin\constantrandom{\mathrm{C}}$ and $\sqgroup{\nicefrac{1}{2},\nicefrac{1}{2}+\epsilon}\notin\constantrandom{\mathrm{C}}$. 
We sketch the proof for $\sqgroup{\nicefrac{1}{2}-\epsilon,\nicefrac{1}{2}}\notin\constantrandom{\mathrm{C}}$.
The proof for $\sqgroup{\nicefrac{1}{2},\nicefrac{1}{2}+\epsilon}\notin\constantrandom{\mathrm{C}}$ is completely analogous. 
Let the multiplier process $\multprocess_{\nicefrac{1}{2},-}$ be defined for all $n\in\naturals$ and $(z_1,\dots,z_{n-1})\in\sits$ by
\begin{equation*}
\multprocess_{\nicefrac{1}{2},-}(z_1,\dots,z_{n-1})
\coloneqq
\begin{cases}
\multprocess_{\nicefrac{1}{2}}(z_1,\dots,z_{n-1})&\text{for $n$ even}\\
1&\text{for $n$ odd},
\end{cases}
\end{equation*}
and similarly for $\multprocess_{\sim\nicefrac{1}{2},-}$. 
Then
\begin{align*}
&\ln\big(\mint[\multprocess_{\nicefrac{1}{2},-}](\xvalto{n})
\mint[\multprocess_{\sim\nicefrac{1}{2},-}](\xvalto{n})\big)
=\sum_{k=1}^{\lfloor\nicefrac{n}{2}\rfloor}\frac{1}{2k+1}\to+\infty
\text{ as $n\to+\infty$}.
\end{align*}
If we now fix any $n\in\naturals$ and $\smash{(z_1,\dots,z_{n-1})\in\sits}$, then
\begin{equation*}
\uex_{[\nicefrac{1}{2}-\epsilon,\nicefrac{1}{2}]}\big(\multprocess_{\nicefrac{1}{2},-}(z_1,\dots,z_{n-1})\big)
=
\begin{cases}
\ex_{\nicefrac{1}{2}}\big(\multprocess_{\nicefrac{1}{2}}(z_1,\dots,z_{n-1})\big)=1
&\text{for $n$ even}\\
\uex_{\nicefrac{1}{2}}(1)=1
&\text{for $n$ odd},
\end{cases}
\end{equation*}
because for even $n$, it follows from $\pmass[n](1)=p_n>1-p_n=\pmass[n](0)$ that also $\multprocess_{\nicefrac{1}{2}}(z_1,\dots,z_{n-1})(1)>\multprocess_{\nicefrac{1}{2}}(z_1,\dots,z_{n-1})(0)$.
Similarly, 
\begin{equation*}
\ex_{p_n}\big(\multprocess_{\sim\nicefrac{1}{2},-}(z_1,\dots,z_{n-1})\big)
=
\begin{cases}
\ex_{p_n}\big(\multprocess_{\sim\nicefrac{1}{2}}(z_1,\dots,z_{n-1})\big)=1
&\text{for $n$ even}\\
\ex_{p_n}(1)=1
&\text{for $n$ odd}.
\end{cases}
\end{equation*}
This tells us that $\multprocess_{\nicefrac{1}{2},-}$ is a supermartingale multiplier for the stationary forecasting system $\constantfrcst{\sqgroup{\nicefrac{1}{2}-\epsilon,\nicefrac{1}{2}}}$, and that $\multprocess_{\sim\nicefrac{1}{2},-}$ is a supermartingale multiplier for the forecasting system $\xmplfrcstsystem$. Since $\multprocess_{\nicefrac{1}{2},-}$ and $\multprocess_{\sim\nicefrac{1}{2},-}$ are clearly computable, the rest of the proof is now similar to that of $\set{\nicefrac{1}{2}}\notin\constantrandom{\mathrm{C}}$.

Finally, we show that for any $\epsilon_1,\epsilon_2\in(0,\nicefrac{1}{2}]$, $I_{\epsilon_1,\epsilon_2}\coloneqq\sqgroup{\nicefrac{1}{2}-\epsilon_1,\nicefrac{1}{2}+\epsilon_2}\in\constantrandom{\mathrm{C}}$.
Assume {\itshape ex absurdo} that $I_{\epsilon_1,\epsilon_2}\notin\constantrandom{\mathrm{C}}$ for some $\epsilon_1,\epsilon_2\in(0,\nicefrac{1}{2}]$, meaning that there is some computable supermartingale multiplier $\multprocess_{\epsilon_1,\epsilon_2}$ for the stationary forecasting system $I_{\epsilon_1,\epsilon_2}$ such that $\mint[\multprocess_{\epsilon_1,\epsilon_2}]$ is unbounded on $\pth=(x_1,\dots,x_n,\dots)$. 
Consider any rational number $\alpha$ such that $0<\alpha\leq\min\set{\epsilon_1,\epsilon_2}$ and let $n_\alpha$ be the smallest natural number $n\in\naturals$ such that $\delta_n\leq\alpha$ and therefore $p_n\in I_{\epsilon_1,\epsilon_2}$, or, using Equations~\eqref{eq:pnoneminus:one} and~\eqref{eq:pnoneminus:two}, let
\begin{equation*}
n_\alpha=
\bigg\lceil-\frac{1}{\ln\frac{1+\sqrt{1-4\alpha^2}}{2}}-1\bigg\rceil.\vspace{6pt}
\end{equation*} 
Then since $\alpha$ is rational and therefore computable, $n_\alpha$ is computable as well. We now consider a new multiplier process $\multprocess$, defined by
\begin{equation*}
\multprocess(z_1,\dots,z_{n-1})
\coloneqq
\begin{cases}
1&\text{if $n<n_\alpha$}\\
\multprocess_{\epsilon_1,\epsilon_2}(z_1,\dots,z_{n-1})
&\text{if $n\geq n_\alpha$}
\end{cases}
\quad\text{for all $n\in\naturals$ and $(z_1\dots,z_{n-1})\in\sits$.}
\end{equation*}
Since $\multprocess_{\epsilon_1,\epsilon_2}$ is computable and $n_\alpha$ is computable, $\multprocess$ is clearly computable. 
Furthermore, since $p_{n}\in I_{\epsilon_1,\epsilon_2}$ for $n\geq n_\alpha$, we find for all $n\in\naturals$ and $(z_1\dots,z_{n-1})\in\sits$ that
\begin{equation*}
\ex_{p_n}\big(\multprocess(z_1,\dots,z_{n-1})\big)
\leq
\begin{cases}
\ex_{p_n}(1)=1
&\text{if $n<n_\alpha$}\\
\uex_{I_{\epsilon_1,\epsilon_2}}\big(\multprocess_{\epsilon_1,\epsilon_2}(z_1,\dots,z_{n-1})\big)\leq1
&\text{if $n\geq n_\alpha$,}
\end{cases}
\end{equation*}
which implies that $\multprocess$ is a computable supermartingale multiplier for the forecasting system $\xmplfrcstsystem$. 
By construction, $\mint$ is unbounded on $\pth$, simply because $\mint[\multprocess_{\epsilon_1,\epsilon_2}]$ is unbounded on $\pth$. 
This contradicts the fact that $\pth$ is computably random for $\xmplfrcstsystem$.
\end{proof}

\end{document}